\documentclass[12pt]{amsart}
\usepackage{amsmath}
\usepackage{amssymb}
\usepackage{amsfonts}

\newtheorem{theorem}{Theorem}
\theoremstyle{plain}

\newtheorem{lemma}{Lemma}

\newtheorem{remark}{Remark}

\numberwithin{equation}{section}
\begin{document}
\title[Limit theorems for Multifractal Products]{Limit theorems for
Multifractal Products of Geometric Stationary Processes}
\author{D.E. Denisov}
\address[D.E. Denisov]{ School of Mathematics \\
University of Manchester, Oxford Road \\
Manchester M13 9PL, UK}
\email{denis.denisov@manchester.ac.uk}
\author{N.N. Leonenko}
\address[N.N. Leonenko]{ School of Mathematics \\
Cardiff University, Senghennydd Road \\
Cardiff CF24 4YH, UK}
\email{LeonenkoN@Cardiff.ac.uk}
\subjclass[2000]{Primary 60G57; 60G10; 60G17}
\keywords{Multifractal products, stationary processes, short-range
dependence, long-range dependence, multifractal scenarios, R\'{e}nyi
function, geometric Gaussian process, geometric Ornstein-Uhlenbeck
processes, L\'{e}vy processes, superpositions, scaling of moments,
log-normal scenario, log-gamma scenario, log normal tempered stable
scenario, log-variance gamma scenario}

\begin{abstract}
We investigate the properties of multifractal products of geometric Gaussian
processes with possible long-range dependence and geometric
Ornstein-Uhlenbeck processes driven by L\'{e}vy motion and their finite and
infinite superpositions. We present the general conditions for the $\mathcal{%
L}_{q}$ convergence of cumulative processes to the limiting processes and
investigate their $q$-th order moments and R\'{e}nyi functions, which are
nonlinear, hence displaying the multifractality of the processes as
constructed. We also establish the corresponding scenarios for the limiting
processes, such as log-normal, log-gamma, log-tempered stable or log-normal
tempered stable scenarios.
\end{abstract}

\maketitle


\section{Introduction}

Multifractal models have been used in many applications in hydrodynamic
turbulence, finance, genomics, computer network traffic, etc. (see, for
example, Kolmogorov 1941, 1962, Kahane 1985, 1987, Novikov 1994, Frisch
1995, Mandelbrot 1997, Falconer 1997, Schertzer \emph{et al.} 1997, Harte
2001, Reidi 2003). There are many ways to construct random multifractal
models ranging from simple binomial cascades to measures generated by
branching processes and the compound Poisson process (Kahane 1985, 1987,
Falconer 1997, Schmitt and Marsan 2001, Harte 2001, Barral and Mandelbrot
2002, 2010, Bacry and Muzy 2003, Riedi 2003, M\"{o}rters and Shieh 2004,
Shieh and Taylor 2002, Schmitt 2003, Schertzer \emph{et al.} 1997, Barall et
al. 2009, Ludena 2008, Jaffard \emph{et al.} 2010). Jaffard (1999) showed
that L\'{e}vy processes (except Brownian motion and Poisson processes) are
multifractal; but since the increments of a L\'{e}vy process are
independent, this class excludes the effects of dependence structures.
Moreover, L\'{e}vy processes have a linear singularity spectrum while real
data often exhibit a strictly concave spectrum.

Anh, Leonenko and Shieh (2008a,b, 2009a,b, 2010) considered multifractal
products of stochastic processes as defined in Kahane (1985, 1987) and
Mannersalo, Norros and Riedi (2002). Especially Anh, Leonenko and Shieh
(2008a) constructed multifractal processes based on products of geometric
Ornstein-Uhlenbeck (OU) processes driven by L\'{e}vy motion with inverse
Gaussian or normal inverse Gaussian distribution. They also described the
behaviour of the $q$-th order moments and R\'{e}nyi functions, which are
nonlinear, hence displaying the multifractality of the processes as
constructed. In these papers a number of scenarios were obtained for $q\in
Q\cap \lbrack 1,2]$, where $Q$ is a set of parameters of marginal
distribution of an OU processes driven by L\'{e}vy motion. The simulations
show that for $q$ outside this range, the scenarios still hold (see Anh,
Leonenko, Shieh and Taufer (2010)). In this paper we present a rigorous
proof of these results and also construct new scenarios which generalize
those corresponding to the inverse Gaussian and normal inverse Gaussian
distributions obtained in Anh and Leonenko (2008), Anh, Leonenko and Shieh
(2008a). We use the theory of OU processes with tempered stable law and
normal tempered stable law for their marginal distributions (see
Barndorff-Nielsen and Shephard 2002 and the references therein). Note that
in their pioneering paper Calvet and Fisher (2002) proposed the simplified
version of the construction of Mannersalo, Norros and Riedi (2002).

The next section recaptures some basic results on multifractal products of
stochastic processes as developed in Kahane (1985, 1987) and Mannersalo,
Norros and Riedi (2002). Section 3 contains the general $\mathcal{L}_{q}$
bounds for cumulative process of multifractal products of stationary
processes. Section 4 establishes the general results on the scaling moments
of multifractal products of geometric OU processes in terms of the marginal
distributions of OU processes and their L\'{e}vy measures. Similar results
for the finite and infinite superpositions of OU processes are proved in the
Section 7. The number of multifractal scenarios with exact forms of 
the scaling function are given in the Sections 8-11.

Our exposition extends results of Mannersalo, Norros and Riedi (2002) on the
basic properties of multifractal products of stochastic processes. We should
also note some related results by Barndorff-Nielsen and Schmiegel (2004) who
introduced some L\'{e}vy-based spatiotemporal models for parametric
modelling of turbulence. Log-infinitely divisible scenarios related to
independently scattered random measures were investigated in Schmitt and
Marsan (2001), Schmitt (2003), Bacry and Muzy (2003), Rhodes and Vargas
(2010), see also their references.

\section{Multifractal products of stochastic processes}

This section recaptures some basic results on multifractal products of
stochastic processes as developed in Kahane (1985, 1987) and Mannersalo,
Norros and Riedi (2002). We provide an interpretation of their conditions
based on the moment generating functions, which is useful for our
exposition. Throughout the text the notation $C,c$ is used for the generic
constants which do not necessarily coincide.

We introduce the following conditions:

\begin{itemize}
\item[\textbf{A}$^{\prime }.$] Let $\Lambda (t),$ $\ t\in \mathbb{R}%
_{+}=[0,\infty ),$ be a measurable, separable, strictly stationary, positive
stochastic process with $\mathrm{E}\Lambda (t)=1$.

We call this process the mother process and consider the following setting:

\item[\textbf{A}$^{\prime \prime }.$] Let $\Lambda(t)=\Lambda
^{(i)},~i=0,1,...$ be independent copies of the mother process $\Lambda ,$
and $\Lambda _{b}^{(i)}$ be the rescaled version of $\Lambda ^{(i)}:$ 
\begin{equation*}
\Lambda _{b}^{(i)}(t)\overset{d}{=}\Lambda ^{(i)}(tb^{i}),\text{\quad }t\in 
\mathbb{R}_{+},\text{\quad }i=0,1,2,\ldots ,
\end{equation*}%
where the scaling parameter $b>1,$ and $\overset{d}{=}$ denotes equality in
finite-dimensional distributions. \ 

Moreover, in the examples, the stationary mother process satisfies the
following conditions:

\item[\textbf{A}$^{\prime \prime \prime }.$] Let $\Lambda (t)=\exp \{X(t)\},$
$t\in \mathbb{R}_{+},$ where $X$ $\left( t\right) $ is a strictly stationary
process, such that there exist a marginal probability density function $\pi
(x)$ and a bivariate probability density function $%
p(x_{1},x_{2};t_{1}-t_{2}) $. Moreover, we assume that the moment generating
function 
\begin{equation}
M(\zeta )=\mathrm{E}\exp \{\zeta X(t)\}  \label{eq:mgf1}
\end{equation}%
and the bivariate moment generating function 
\begin{equation}
M(\zeta _{1},\zeta _{2};t_{1}-t_{2})=\mathrm{E}\exp \{\zeta
_{1}X(t_{1})+\zeta _{2}X(t_{2})\}  \label{eq:mgf2}
\end{equation}%
exist.
\end{itemize}

The conditions \textbf{A}$^{\prime }$-\textbf{A}$^{\prime \prime \prime }$
yield%
\begin{equation*}
\mathrm{E}\Lambda _{b}^{(i)}(t)=M(1)=1;\text{\textrm{Var}}\Lambda
_{b}^{(i)}(t)=M(2)-1=\sigma _{\Lambda }^{2}<\infty ;
\end{equation*}%
\begin{equation*}
\text{\textrm{Cov}}(\Lambda _{b}^{(i)}(t_{1}),\Lambda
_{b}^{(i)}(t_{2}))=M(1,1;(t_{1}-t_{2})b^{i})-1,~b>1.
\end{equation*}%
We define the finite product processes 
\begin{equation}
\Lambda _{n}(t)=\prod_{i=0}^{n}\Lambda _{b}^{(i)}(t)=\exp \left\{
\sum_{i=0}^{n}X^{(i)}(tb^{i})\right\} ,t\in \lbrack 0,1],  \label{2.1}
\end{equation}%
and the cumulative processes 
\begin{equation}
A_{n}(t)=\int_{0}^{t}\Lambda _{n}(s)ds,\text{\quad }n=0,1,2,\ldots ,t\in
\lbrack 0,1],  \label{2.2}
\end{equation}%
where $X^{(i)}(t),i=0,...,n,....,$ are independent copies of a stationary
process $X(t),t\geq 0.$

We also consider the corresponding positive random measures defined on Borel
sets $B$ of $\mathbb{R}_{+}:$%
\begin{equation}
\mu _{n}(B)=\int_{B}\Lambda _{n}(s)ds,\text{\quad }n=0,1,2,\ldots
\label{2.3}
\end{equation}%
Kahane (1987) proved that the sequence of random measures $\mu _{n}$
converges weakly almost surely to a random measure $\mu $. \ Moreover, given
a finite or countable family of Borel sets $B_{j}$ on $\mathbb{R}_{+}$, it
holds that $\lim_{n\rightarrow \infty }\mu _{n}(B_{j})=\mu (B_{j})$ for all $%
j$ with probability one$.$ The almost sure convergence of $A_{n}\left(
t\right) $ in countably many points of $\mathbb{R}_{+}$ can be extended to
all points in $\mathbb{R}_{+}$ if the limit process $A\left( t\right) $ is
almost surely continuous. In this case, $\lim_{n\rightarrow \infty
}A_{n}(t)= $ $A(t)$ with probability one for all $t\in \mathbb{R}_{+}.$ As
noted in Kahane (1987), there are two extreme cases: (i) $%
A_{n}(t)\rightarrow A(t)$ in $\mathcal{L}_{1}$ for each given $t$, in which
case $A(t)$ is not almost surely zero and and is said to be fully active
(non-degenerate) on $\mathbb{R}_{+}$; (ii) $A_{n}(1)$ converges to $0$
almost surely, in which case $A(t)$ is said to be degenerate on $\mathbb{R}%
_{+}$. Sufficient conditions for non-degeneracy and degeneracy in a general
situation \ and relevant examples are provided in Kahane (1987) (Eqs. (18)
and (19) respectively.) The condition for complete degeneracy is detailed in
Theorem 3 of Kahane (1987). In our work we present general conditions for
non-degeneracy in Theorem~\ref{thm:scaling}.

The R\'{e}nyi function of a random measure $\mu $, also known as the
deterministic partition function, is defined for $t\in \lbrack 0,1]$ as%
\begin{equation*}
T(q)=\underset{n\rightarrow \infty }{\lim \inf }\frac{\log \mathrm{E}%
\sum_{k=0}^{2^{n}-1}\mu ^{q}\left( I_{k}^{(n)}\right) }{\log \left\vert
I_{k}^{(n)}\right\vert }=\underset{n\rightarrow \infty }{\lim \inf }\left( -%
\frac{1}{n}\right) \log _{2}\mathrm{E}\sum_{k=0}^{2^{n}-1}\mu ^{q}\left(
I_{k}^{(n)}\right) ,
\end{equation*}%
where $I_{k}^{(n)}=\left[ k2^{-n},(k+1)2^{-n}\right] ,$\quad $k=0,1,\ldots
,2^{n}-1,$ $\left\vert I_{k}^{(n)}\right\vert $ is its length, and $\log
_{b} $ is log to the base $b.$


In the present paper we establish convergence 
\begin{equation}
A_{n}(t)\overset{\mathcal{L}_{q}}{\rightarrow }A(t),\quad n\rightarrow
\infty .  \label{e319}
\end{equation}%
For the limiting process we show that for some constants $\overline{C}$ and $%
\underline{C}$, 
\begin{equation}
\underline{C}t^{q-\log _{b}\mathrm{E}\Lambda ^{q}\left( t\right) }\leqslant 
\mathrm{E}A^{q}(t)\leqslant \overline{C}t^{q-\log _{b}\mathrm{E}\Lambda
^{q}\left( t\right) },  \label{2.6}
\end{equation}%
which will be written as 
\begin{equation*}
\mathrm{E}A^{q}(t)\sim t^{q-\log _{b}\mathrm{E}\Lambda ^{q}\left( t\right) }.
\end{equation*}%
This allows us to find the scaling function 
\begin{equation}
\varsigma (q)=q-\log _{b}\mathrm{E}\Lambda ^{q}\left( t\right) =q-\log
_{b}M(q).  \label{2.7}
\end{equation}%
As is shown in \cite{LS2013} for the exponentially decreasing correlations
and $q\in \lbrack 1,2]$ there is a connection between R\'{e}nyi function and
the scaling function given by 
\begin{equation}
T(q)=\varsigma (q)-1.  \label{2.7a}
\end{equation}%
The exact conditions are stated in Theorem~\ref{Th3.1} and Theorem~\ref%
{thm:scaling}.

An important contribution of our paper is that we proved (\ref{e319}) for
general $q>0$. In comparison, in Mannersalo, Norros and Riedi (2002)
convergence (\ref{e319}) was shown for $q\in [1,2]$ under an additional
assumption $A(t)\in \mathcal{L}_q$. Additionally we simplified significantly
the conditions under which equations (\ref{e319}) and (\ref{2.6}) hold.
Finally we provide a number of scenarios where scaling function can be
written explicitly.


\section{$\mathcal{L}_{q}$ convergence: general bound}

This section contains a generalisation of the basic results on multifractal
products of stochastic processes developed in Kahane (1985, 1987) and
Mannersalo, Norros and Riedi\emph{\ }(2002).

Consider the cumulative process $A_{n}(t)$ defined in (\ref{2.2}). For fixed 
$t$, the sequence $\{A_{n}(t),\mathcal{F}_{n}\}_{n=0}^{\infty }$ is a
martingale. It is well known that for $q>1$, $\mathcal{L}_{q}$ convergence
is equivalent to the finiteness of 
\begin{equation*}
\sup_{n}\mathrm{E}A_{n}^{q}(t)<\infty .
\end{equation*}

\subsection{$\mathcal{L}_{2}$ convergence}

First we consider a simpler case $q=2$, which was studied in Mannersalo 
\emph{et al. }(2002). The proof in the general case uses the same idea but
is more complicated.

We have, 
\begin{equation*}
\mathrm{E}A_{n}^{2}(t)=\mathrm{E}\int_{0}^{t}\int_{0}^{t}\Lambda
_{n}(s_{1})\Lambda
_{n}(s_{2})ds_{1}ds_{2}=\int_{0}^{t}\int_{0}^{t}\prod_{i=0}^{n}\mathrm{E}%
\Lambda ^{(i)}(s_{1})\Lambda ^{(i)}(s_{2})ds_{1}ds_{2}.
\end{equation*}%
The process $\Lambda ^{(i)}$ is stationary. Therefore, 
\begin{eqnarray*}
\mathrm{E}A_{n}^{2}(t) &=&2\int_{0}^{t}\int_{s_{1}}^{t}\prod_{i=0}^{n}%
\mathrm{E}\Lambda ^{(i)}(0)\Lambda ^{(i)}(s_{2}-s_{1})ds_{1}ds_{2} \\
&=&2\int_{0}^{t}\int_{0}^{t-s_{1}}\prod_{i=0}^{n}\rho
(b^{i}(s_{2}-s_{1})ds_{1}ds_{2}\leq 2t\int_{0}^{t}\prod_{i=0}^{n}\rho
(b^{i}u)du,
\end{eqnarray*}%
where 
\begin{equation}
\rho (u)=\mathrm{E}\Lambda (0)\Lambda (u).  \label{ro1}
\end{equation}%
Hence, to show $\mathcal{L}_{2}$ convergence it is sufficient to show that 
\begin{equation*}
\sup_{n}\int_{0}^{t}\prod_{i=0}^{n}\rho (b^{l}u)du<\infty .
\end{equation*}

\begin{theorem}
\label{lem1} Assume that $\rho (u)$ as defined in (\ref{ro1}) is monotone
decreasing in $u$, 
\begin{equation}
b>\mathrm{E}\Lambda (0)^{2}  \label{eq101}
\end{equation}%
and 
\begin{equation}
\sum_{i=0}^{\infty }(\rho (b^{i})-1)<\infty .  \label{eq100}
\end{equation}%
Then $A_{n}(t)$ converges in $\mathcal{L}_{2}$ (and hence in $\mathcal{L}_{q}$ for $q\in[0,2]$) for every fixed $t\in[0,1]$.
\end{theorem}

\textsc{Proof.} 
First note that  $\mathcal{L}_{2}$ convergence implies $\mathcal{L}_{q}$ convergence for all $q\in [0,2]$. 
This follows from the inequality $\mathrm{E}|A_n(t)-A(t)|^s\le (\mathrm{E}|A_n(t)-A(t)|^2)^{s/2}$ valid for any $s\le 2$. In turn the latter inequality follows from the Jensen inequality. 

Without loss of generality let $t=1$. Let $n(u)=[-\log
_{b}u]$ be the integer part of $-\log _{b}u$. Then, using monotonicity of $%
\rho $ we obtain 
\begin{equation*}
\prod_{i=0}^{n}\rho (b^{i}u)\leq \rho
(0)^{n(u)}\prod_{i=n(u)}^{n}\rho (b^{i}u).
\end{equation*}%
Using monotonicity of $\rho $ again, 
\begin{equation*}
\prod_{i=n(u)}^{n}\rho (b^{i}u)\leq \prod_{i=0}^{n-n(u)}\rho
(b^{i+n(u)}u)\leq \Pi:=\prod_{i=0}^{\infty }\rho (b^{i}).
\end{equation*}%
Constant $\Pi $ is finite due to the condition (\ref{eq100}). For
sufficiently small $\delta \in (0,1)$, by the condition (\ref{eq101}), $%
b^{1-\delta }>\rho (0)=\mathrm{E}\Lambda (0)^{2}$ . Therefore, 
\begin{equation*}
\sup_{n}\int_{0}^{1}\prod_{i=0}^{n}\rho (b^{i}u)du\leq \Pi \int_{0}^{1}\rho
(0)^{n(u)}du\leq \Pi \int_{0}^{1}b^{(1-\delta )n(u)}du\leq \Pi \int_{0}^{1}%
\frac{1}{u^{1-\delta }}du<\infty .
\end{equation*}%
The proof of Theorem \ref{lem1} is complete.

\qed

\subsection{$\mathcal{L}_{q}$ convergence for $q>2$}

Now we are going to consider $q>2$. Now we assume additionally that $%
A_{n}(t) $ is a cadlag process. Also, we strengthen condition \eqref{eq100}.
For that let 
\begin{equation}
\rho (u_{1},\ldots ,u_{q-1})=\mathrm{E}\Lambda (0)\Lambda (u_{1})\ldots
\Lambda (u_{1}+\cdots +u_{q-1})  \label{RO}
\end{equation}
We require that the function $\rho (u_{1},\ldots ,u_{q-1})$ satisfies
certain mixing conditions. Namely, let $m<q-1$ and ${\mathcal{C}}%
=\{i_1,\ldots i_m\}$ be a subset of indices ordered in the increasing order $%
1\le i_{1}< \ldots <i_{m}\le q-1$. Consider the vector $(u_1,\ldots,u_{q-1})$
such that $u_j=A$ if $j\in \mathcal{C}$ and $u_j=0$ otherwise. Then we
assume that for any set $\mathcal{C}$ the following mixing condition holds 
\begin{equation}
\lim_{A\rightarrow \infty }\rho (u_{1},\ldots ,u_{q-1})=\mathrm{E}\Lambda
(0)^{i_{1}}\mathrm{E}\Lambda (0)^{i_{2}-i_{1}}\cdot \ldots \cdot \mathrm{E}%
\Lambda (0)^{q-i_{m}}.  \label{eq:mixing}
\end{equation}

The starting point is the equality 
\begin{align}
\mathrm{E}A_{n}^{q}(t)& =\mathrm{E}\int_{0}^{t}\int_{0}^{t}\ldots
\int_{0}^{t}\Lambda _{n}(s_{1})\Lambda _{n}(s_{2})\ldots \Lambda
_{n}(s_{q})ds_{1}ds_{2}\ldots ds_{q}  \notag \\
& =q!\int_{0<s_{1}<\ldots <s_{q}<t}\mathrm{E}\Lambda _{n}(s_{1})\Lambda
_{n}(s_{2})\ldots \Lambda _{n}(s_{p})ds_{1}ds_{2}\ldots ds_{q}.  \label{eqsp}
\end{align}%
First we make change of variables%
\begin{equation*}
u_{0}=s_{1},u_{1}=s_{2}-s_{1},\ldots u_{q-1}=s_{q}-s_{q-1},
\end{equation*}
which transforms equality (\ref{eqsp}) into 
\begin{align*}
\mathrm{E}A_{n}^{q}(t)& =q!\int_{0<u_{0},\ldots ,u_{q-1}}^{u_{0}+\cdots
+u_{q-1}\leq t}\mathrm{E}\Lambda _{n}(u_{0})\Lambda _{n}(u_{0}+u_{1})\ldots
\Lambda _{n}(u_{0}+\cdots +u_{q-1})du_{0}\ldots du_{q-1} \\
& \leq q!\int_{0<u_{0},\ldots ,u_{q-1}<t}\mathrm{E}\Lambda
_{n}(u_{0})\Lambda _{n}(u_{0}+u_{1})\ldots \Lambda _{n}(u_{0}+\cdots
+u_{q-1})du_{0}du_{1}\ldots du_{q-1} \\
& =q!\int_{0<u_{1},\ldots ,u_{q-1}<t}\mathrm{E}\Lambda _{n}(0)\Lambda
_{n}(u_{1})\ldots \Lambda _{n}(u_{1}+\cdots +u_{q-1})du_{1}\ldots du_{q-1},
\end{align*}%
where we used stationarity of the process $\Lambda (t)$ to obtain the latter
inequality. Thus it is sufficient to prove that 
\begin{equation}
\sup_{n}\int_{0<u_{1},\ldots ,u_{q-1}<t}\prod_{l=0}^{n}\rho
(b^{l}u_{1},\ldots ,b^{l}u_{q-1})du_{1}\ldots du_{q-1}<\infty .
\label{eq:eq_N2}
\end{equation}

We are ready now to state the main result of this section.

\begin{theorem}
\label{Th3.1}Suppose that conditions \textbf{A}$^{\prime }$-\textbf{A}$%
^{\prime \prime }$ hold. Assume that $\rho (u_{1},\ldots ,u_{q-1})$ defined
in (\ref{RO}) is monotone decreasing in all variables. Let 
\begin{equation}
b^{q-1}>\mathrm{E}\Lambda (0)^{q}  \label{eq:b.greater}
\end{equation}%
for some integer $q\ge 2,$ and 
\begin{equation}
\sum_{n=1}^{\infty }\left( \rho (b^{n},\ldots ,b^{n})-1\right) <\infty .
\label{eq:sum_rho_finite}
\end{equation}%
Finally assume that the mixing condition (\ref{eq:mixing}) holds. Then, 
\begin{equation}
\mathrm{E}A(t)^{q}<\infty ,  \label{eq:lq.finite}
\end{equation}%
and $A_{n}(t)$ converges to $A(t)$ in $\mathcal{L}_{q}$ (and hence in $\mathcal{L}_{\widetilde q}$ for $\widetilde q\in[0,q]$).
\end{theorem}

\textit{Proof of Theorem~\ref{Th3.1}}. 
As above  $\mathcal{L}_{q}$ convergence implies $\mathcal{L}_{\widetilde q}$ convergence for all 
$\widetilde q\in [0,q]$. 
This follows from the inequality $\mathrm{E}|A_n(t)-A(t)|^{\widetilde q}\le (\mathrm{E}|A_n(t)-A(t)|^q)^{\widetilde q/q}$ valid for any $\widetilde q\le q$. 

It is sufficient to prove that
equation (\ref{eq:eq_N2}) holds. To simplify notation we put $t=1$. First
represent the integral in (\ref{eq:eq_N2}) as the sum of the integrals over
different regions 
\begin{multline}
\int_{0\leq u_{1},\ldots ,u_{q-1}\leq 1}\prod_{l=0}^{n}\rho
(b^{l}u_{1},\ldots ,b^{l}u_{q-1})du_{1}\ldots du_{q-1}
\label{eq:int_sum_regions} \\
=\sum_{i_{1},\ldots i_{q-1}}\int_{0\le u_{i_{1}}\leq u_{i_{2}}\leq
\ldots\leq u_{i_{q-1}}\leq 1}\prod_{l=0}^{n}\rho (b^{l}u_{1},\ldots
,b^{l}u_{q-1})du_{1}\ldots du_{q-1},
\end{multline}%
where the sum is taken over all possible permutations of numbers $%
(1,2,\ldots ,q-1).$ Next we are going to bound the integrals on these
separate regions. Put 
\begin{equation*}
u_{(1)}=u_{i_{1}},u_{(2)}=u_{i_{2}},\ldots ,u_{(q-1)}=u_{i_{q-1}}.
\end{equation*}%
Fix a large number $A\geq 1$ which we define later and define an auxiliary
function $n(u)=-[\log _{b}u/A]$. Note that this function is non-negative for 
$u\leq 1$. Now let 
\begin{equation*}
l_{1}=n(u_{(1)}),l_{2}=n(u_{(2)}),\ldots ,l_{q-1}=n(u_{(q-1)}).
\end{equation*}%
These numbers are decreasing 
\begin{equation}
l_{1}\geq l_{2}\geq \ldots \geq l_{q-1}.  \label{eq:l.decreasing}
\end{equation}%
Then we can split the product as 
\begin{equation}
\prod_{l=0}^{n}\rho (b^{l}u_{1},\ldots
,b^{l}u_{q-1})=\prod_{l=0}^{l_{q-1}-1}\prod_{l=l_{q-1}}^{l_{q-2}-1}\ldots
\prod_{l=l_{2}}^{l_{1}-1}\prod_{l=l_{1}}^{n}\rho (b^{l}u_{1},\ldots
,b^{l}u_{q-1}).  \label{eq:split.product}
\end{equation}%
Further, using monotonicity of the function $\rho $ we can estimate for $%
l<l_{q-1}$, 
\begin{equation*}
\rho (b^{l}u_{1},\ldots ,b^{l}u_{q-1})\leq \rho (0,\ldots ,0)=\mathrm{E}%
\Lambda (0)^{q}.
\end{equation*}%
For $l\in \lbrack l_{q-1},l_{q-2})$, we have 
\begin{equation*}
\rho (b^{l}u_{1},\ldots ,b^{l}u_{q-1})\leq \rho (0,\ldots ,0,A,0\ldots ,0),
\end{equation*}%
where $i_{q-1}$th argument of the function $\rho $ is equal to $A$ and all
other arguments are equal to $0$. Indeed this holds due to the fact that for 
$l>l_{q-1}$ 
\begin{equation*}
b^{l}u_{(q-1)}\geq b^{l_{q-1}}u_{(q-1)}\geq \frac{A}{u_{(q-1)}}u_{(q-1)}=A
\end{equation*}%
and the monotonicity of the function $\rho $. Here recall that $u_{(q-1)}$
corresponds to $u_{i_{q-1}}$. Fix a small number $\delta $ which we define
later. Now we can note that mixing condition (\ref{eq:mixing}) implies that 
\begin{equation*}
\lim_{A\rightarrow \infty }\rho (0,\ldots ,0,A,0,\ldots ,0)=\mathrm{E}%
\Lambda (0)^{i_{q-1}}\mathrm{E}\Lambda (0)^{q-i_{q-1}}
\end{equation*}%
Hence we can pick $A=A(\delta )$ sufficiently large to ensure that 
\begin{equation*}
\rho (0,\ldots ,0,A,0,\ldots ,0)\leq (1+\delta )\mathrm{E}\Lambda
(0)^{i_{q-1}}\mathrm{E}\Lambda (0)^{q-i_{q-1}}.
\end{equation*}%
Function $g(x)=\ln \mathrm{E}\Lambda (0)^{x}$ is convex. Hence we can apply
Karamata majorization inequality \cite{Kar32} to obtain that 
\begin{equation*}
g(i_{q-1})+g(q-i_{q-1})\leq g(q-1)+g(1).
\end{equation*}%
Therefore, 
\begin{equation*}
\mathrm{E}\Lambda (0)^{i_{q-1}}\mathrm{E}\Lambda (0)^{q-i_{q-1}}\leq \mathrm{%
E}\Lambda (0)^{q-1}\mathrm{E}\Lambda (0)=\mathrm{E}\Lambda (0)^{q-1}
\end{equation*}%
and 
\begin{equation*}
\rho (0,\ldots ,0,A,0,\ldots ,0)\leq (1+\delta )\mathrm{E}\Lambda (0)^{q-1}.
\end{equation*}%
Similarly, for $l\in \lbrack l_{q-2},l_{q-3}]$, we have 
\begin{equation*}
\rho (b^{l}u_{1},\ldots ,b^{l}u_{q-1})\leq \rho (0,\ldots ,0,A,\ldots
,0,A,0\ldots ,0),
\end{equation*}%
where the arguments of the function $\rho $ are equal to $0$ except
arguments $i_{q-1}$ and $i_{q-2}$ which are equal to $A$. Applying the
mixing condition and increasing $A$ if necessary we can ensure that for $%
l\in \lbrack l_{q-2},l_{q-3}]$, 
\begin{equation*}
\rho (b^{l}u_{1},\ldots ,b^{l}u_{q-1})\leq (1+\delta )\mathrm{E}\Lambda
(0)^{a}\mathrm{E}\Lambda (0)^{b-a}\mathrm{E}\Lambda (0)^{q-b},
\end{equation*}%
where $a=\min (i_{q-2},i_{q-1})$, $b=\max (i_{q-2},i_{q-1})$. We apply now
Karamata's majorisation inequality twice. First application of the
inequality gives 
\begin{equation*}
\mathrm{E}\Lambda (0)^{a}\mathrm{E}\Lambda(0)^{b-a}\le\mathrm{E}\Lambda
(0)^{b-1}.
\end{equation*}%
Second application of Karamata's inequality gives 
\begin{equation*}
\mathrm{E}\Lambda (0)^{b-1}\mathrm{E}\Lambda (0)^{q-b}\leq \mathrm{E}\Lambda
(0)^{q-2}.
\end{equation*}%
Hence, for $l\in \lbrack l_{q-2},l_{q-3})$ and sufficiently large $A$, 
\begin{equation*}
\rho (b^{l}u_{1},\ldots ,b^{l}u^{q-1})\leq (1+\delta )\mathrm{E}\Lambda
(0)^{a}\mathrm{E}\Lambda (0)^{b-a}\mathrm{E}\Lambda (0)^{q-b}\leq (1+\delta)%
\mathrm{E}\Lambda (0)^{q-2}.
\end{equation*}%
In exactly the same manner, using the mixing conditions and Karamata's
majorisation inequality one can obtain for $l\in \lbrack l_{j},l_{j-1})$ and 
$j=q-1,q-2,\ldots ,2$ 
\begin{equation*}
\rho (b^{l}u_{1},\ldots ,b^{l}u^{q-1})\leq (1+\delta)\mathrm{E}\Lambda
(0)^{j}.
\end{equation*}%
Hence, 
\begin{align}
\prod_{l=0}^{l_{1}-1}\rho (b^{l}u_{1},\ldots
,b^{l}u_{q-1})&=\prod_{l=0}^{l_{q-1}-1}\prod_{l=l_{q-1}}^{l_{q-2}-1}\ldots
\prod_{l=l_{2}}^{l_{1}-1}\rho (b^{l}u_{1},\ldots ,b^{l}u_{q-1})
\label{eq:est.prod} \\
&\leq (1+\delta )^{l_{1}}\prod_{i=2}^{q}\prod_{l=l_{i}}^{l_{i-1}-1}\mathrm{E}%
\Lambda (0)^{i}=(1+\delta )^{l_{1}}\prod_{i=2}^{q}\left( \mathrm{E}\Lambda
(0)^{i}\right) ^{l_{i-1}-l_{i}},  \notag
\end{align}%
where $l_{q}=0$. Rearranging the terms we can represent this product in a
slightly different form 
\begin{equation}
\prod_{i=2}^{q}\left( \mathrm{E}\Lambda (0)^{i}\right)
^{l_{i-1}-l_{i}}=\prod_{i=1}^{q-1}\left( \frac{\mathrm{E}\Lambda (0)^{i+1}%
\mathrm{E}\Lambda (0)^{i-1}}{(\mathrm{E}\Lambda (0)^{i})^{2}}\right)
^{l_{q-1}+\cdots +l_{i}}  \label{eq:prod.diff.form}
\end{equation}%
Now one can note that since $l_{i}$ are decreasing, see (\ref%
{eq:l.decreasing}), 
\begin{equation*}
l_{q-1}+\cdots +l_{i}\leq \frac{q-i}{q-1}(l_{1}+\cdots +l_{q-1}),
\end{equation*}%
for any $i=1,\ldots ,q-1$. Indeed, the latter inequality is equivalent to 
\begin{equation*}
(i-1)(l_{q-1}+\cdots +l_{i})\leq (q-i)(l_{i-1}+\cdots +l_{1}),
\end{equation*}%
which follows from 
\begin{equation*}
\frac{l_{q-1}+\cdots +l_{i}}{q-i}\leq l_{i}\leq l_{i-1}\leq \frac{%
l_{i-1}+\cdots +l_{1}}{i-1}.
\end{equation*}%
In addition, by the Karamata's majorization inequality, 
\begin{equation*}
\frac{\mathrm{E}\Lambda (0)^{i+1}\mathrm{E}\Lambda (0)^{i-1}}{(\mathrm{E}%
\Lambda (0)^{i})^{2}}>1.
\end{equation*}%
Therefore, 
\begin{equation*}
\left( \frac{\mathrm{E}\Lambda (0)^{i+1}\mathrm{E}\Lambda (0)^{i-1}}{(%
\mathrm{E}\Lambda (0)^{i})^{2}}\right) ^{l_{q-1}+\cdots +l_{i}}\leq \left( 
\frac{\mathrm{E}\Lambda (0)^{i+1}\mathrm{E}\Lambda (0)^{i-1}}{(\mathrm{E}%
\Lambda (0)^{i})^{2}}\right) ^{\frac{q-i}{q-1}(l_{1}+\ldots +l_{q-1})}
\end{equation*}

Hence we can continue (\ref{eq:prod.diff.form}) as follows 
\begin{multline}
\prod_{i=2}^{q}\left( \mathrm{E}\Lambda (0)^{i}\right) ^{l_{i-1}-l_{i}}\leq
\prod_{i=2}^{q-1}\left( \frac{\mathrm{E}\Lambda (0)^{i+1}\mathrm{E}\Lambda
(0)^{i-1}}{(\mathrm{E}\Lambda (0)^{i})^{2}}\right) ^{\frac{q-i}{q-1}%
(l_{1}+\ldots +l_{q-1})}  \label{eq:product.est.final} \\
=\left( \mathrm{E}\Lambda (0)^{q}\right) ^{\frac{l_{1}+\cdots +l_{q-1}}{q-1}%
}\prod_{i=2}^{q-3}\left( \mathrm{E}\Lambda (0)^{i}\right) ^{\frac{%
q-i+1-2(q-i)+q-i-1}{q-1}}=\left( \mathrm{E}\Lambda (0)^{q}\right) ^{\frac{%
l_{1}+\cdots +l_{q-1}}{q-1}}.
\end{multline}

Plugging the latter estimate in (\ref{eq:est.prod}) we arrive at 
\begin{equation*}
\prod_{l=0}^{l_{1}}\rho (b^{l}u_{1},\ldots ,b^{l}u_{q-1})\leq (1+\delta
)^{l_{1}}\left( \mathrm{E}\Lambda (0)^{q}\right) ^{\frac{l_{1}+\cdots
+l_{q-1}}{q-1}}.
\end{equation*}%
We can now make use of the condition \eqref{eq:b.greater} and by taking $%
\delta $ sufficiently small we can ensure that 
\begin{equation}
\prod_{l=0}^{l_{1}}\rho (b^{l}u_{1},\ldots ,b^{l}u_{q-1})\leq
b^{(1-\varepsilon )(l_{1}+\cdots +l_{q-1})}=(u_{1}u_{2}\ldots
u_{q-1})^{-1+\varepsilon }A^{q(1-\varepsilon )}  \label{eq:est.prod.l2}
\end{equation}%
for some small $\varepsilon >0.$ We are left to estimate the product $%
\prod_{l=l_{1}}^{n}$ uniformly in $n$. For that we are going to use
finiteness of the series in (\ref{eq:sum_rho_finite}). First note that for $%
l\geq l_{1}$, $b^{l}u_{j}\geq Ab^{l-l_{1}}.$ Then, by monotonicity of the
function $\rho $, uniformly in $n$, for some $C>0$ 
\begin{align}
\prod_{l=l_{1}}^{n}\rho (b^{l}u_{1},\ldots ,b^{l}u_{q-1})& \leq
\prod_{l=l_{1}}^{n}\rho (b^{l-l_{1}}A,\ldots ,b^{l-l_{1}}A)  \notag \\
& \leq \prod_{l=0}^{\infty }\rho (b^{l},\ldots ,\ldots ,b^{l})<C,
\label{eq:est.prod.l3}
\end{align}%
according to the finiteness of the series. Together (\ref{eq:est.prod.l2})
and (\ref{eq:est.prod.l3}) give us

\begin{multline}
\int_{0\leq u_{1},\ldots ,u_{q-1}\leq 1}\prod_{l=0}^{n}\rho
(b^{l}u_{1},\ldots ,b^{l}u_{q-1})du_{1}\ldots du_{q-1}
\label{eq:int_sum_regions.final} \\
=\sum_{i_{1},\ldots i_{q-1}}\int_{0<u_{i_{1}}\leq u_{i_{2}}\leq
u_{i_{q-1}}\leq 1}\prod_{l=0}^{n}\rho (b^{l}u_{1},\ldots
,b^{l}u_{q-1})du_{1}\ldots du_{q-1} \\
\leq C\sum_{i_{1},\ldots i_{q-1}}\int_{0<u_{i_{1}}\leq u_{i_{2}}\leq
u_{i_{q-1}}\leq 1}(u_{1}u_{2}\ldots u_{q-1})^{-1+\varepsilon }du_{1}\ldots
du_{q-1} \\
=C\int_{0\leq u_{1},\ldots ,u_{q-1}\leq 1}(u_{1}u_{2}\ldots
u_{q-1})^{-1+\varepsilon }du_{1}\ldots du_{q-1}
\end{multline}%
which immediately gives a finite bound for $\mathrm{E}A_{n}(1)^{q}$ uniform
in $n$.

\qed

\begin{remark}
It is not difficult to show that \eqref{eq:b.greater} is sharp. Indeed
suppose that 
\begin{equation*}
b^{q-1}<\mathrm{E}\Lambda (0)^{q}
\end{equation*}%
and that $\rho (u_{1},\ldots ,u_{q-1})$ is continuous at $(0,\ldots ,0$).
Then, for $\varepsilon >0$, 
\begin{align*}
\mathrm{E}A_{n}^{q}(t)& =q!\int_{0<u_{0},\ldots ,u_{q-1}}^{u_{0}+\cdots
+u_{q-1}\leq t}\mathrm{E}\Lambda _{n}(0)\Lambda _{n}(u_{1})\ldots \Lambda
_{n}(u_{1}+\cdots +u_{q-1})du_{0}\ldots du_{q-1} \\
& =q!\int_{0<u_{0},\ldots ,u_{q-1}}^{u_{0}+\cdots +u_{q-1}\leq
t}\prod_{l=0}^{n}\rho (b^{l}u_{1},\ldots ,b^{l}u_{q-1})du_{0}\ldots du_{q-1}
\\
& \geq q!\int_{0<u_{0}<1/2,0<u_{1}\ldots ,u_{q-1}\leq \varepsilon
/b^{n}}\prod_{l=0}^{n}\rho (b^{l}u_{1},\ldots ,b^{l}u_{q-1})du_{0}\ldots
du_{q-1} \\
& \geq \frac{q!}{2}\int_{0<u_{1}\ldots ,u_{q-1}\leq \varepsilon
/b^{n}}\prod_{l=0}^{n}\rho (\varepsilon ,\ldots ,\varepsilon )du_{1}\ldots
du_{q-1} =\frac{q!}{2}\varepsilon ^{q-1}\left( \frac{\rho (\varepsilon
,\ldots ,\varepsilon )}{b^{q-1}}\right) ^{n}
\end{align*}

Since $\rho (\varepsilon ,\ldots ,\varepsilon )$ can be made arbitrarily
close to $\rho (0,\ldots ,0)=\mathrm{E}\Lambda (0)^{q}$, then, for
sufficiently small $\varepsilon >0$, $\rho (\varepsilon ,\ldots ,\varepsilon
)>b^{q-1},$ and 
\begin{equation*}
\mathrm{E}A_{n}(t)\geq \frac{q!}{2}\varepsilon ^{q-1}\left( \frac{\rho
(\varepsilon ,\ldots ,\varepsilon )}{b^{q-1}}\right) ^{n}\rightarrow \infty,
\quad n\rightarrow \infty .
\end{equation*}
\end{remark}

\section{ Scaling of moments}

\label{section_scaling} The aim of this Section is to establish the scaling
property (\ref{2.6}). For $q>1$ let 
\begin{equation}
\rho _{q}(s)=\inf_{u\in \lbrack 0,1]}\left( \frac{\mathrm{E}\Lambda
(0)^{q-1}\Lambda (su)}{\mathrm{E}\Lambda (0)^{q}}-1\right) .  \label{eq_q1}
\end{equation}%
Note that $\rho _{q}(s)\leq 0$. For $q\in (0,1)$ let 
\begin{equation}
\rho _{q}(s)=\sup_{u\in \lbrack 0,1]}\left( \frac{\mathrm{E}\Lambda
(0)^{q-1}\Lambda (su)}{\mathrm{E}\Lambda (0)^{q}}-1\right) .  \label{eq_q2}
\end{equation}%
For $q\leq 1$ it is easy to see that $\rho _{q}(s)\geq 0.$

\begin{theorem}
\label{thm:scaling} Assume that $A(t)\in \mathcal{L}_{q}$ and $\rho _{q}(s)$
defined in (\ref{eq_q1}) and (\ref{eq_q2}) is such that 
\begin{equation}  \label{eq:ass_scaling}
\sum_{n=1}^{\infty }|\rho _{q}(b^{-n})|<\infty .
\end{equation}%
Then, 
\begin{equation}  \label{eq:mf.bounds}
\mathrm{E}A^{q}(t)\sim t^{q-\log _{b}\mathrm{E}\Lambda ^{q}\left( t\right)
},\quad t\in [0,1].
\end{equation}
and process $A(t)$ is non-degenerate, that is $\mathbb{P}(A(t)>0)>0.$
\end{theorem}

\textit{Proof of Theorem~\ref{thm:scaling}}

Our strategy in proving of (\ref{eq:mf.bounds}) is to use martingale
properties of the sequence $A_{n}(t)$. We concentrate mainly on $q>1$, as
the case $q<1$ is symmetric. For the upper bound we obtain uniform in $n$
bounds from above for $\mathrm{E}A_{n}(t)^{q}$. Then, since $A_{n}(t)$
converges to $A(t) $ in $\mathcal{L}_{q}$, the same estimates hold for $%
\mathrm{E}A(t)^{q}$. For the lower bound, we use the fact that as $%
A_{n}(t)\in \mathcal{L}_{q}$ for $q>1$ the martingale $A_{n}(t)$ is
closable. Hence it can be represented as $A_{n}(t)=\mathrm{E}%
(A(t)|A_{1}(t),\ldots ,A_{n}(t))$. Therefore, for $q>1 $, by the conditional
Jensen inequality, 
\begin{align*}
\mathrm{E}A_{n}(t)^{q}& =\mathrm{E}(\mathrm{E}(A(t)|A_{1}(t),\ldots
,A_{n}(t)))^{q} \\
& \leq \mathrm{E}(\mathrm{E}(A(t)^{q}|A_{1}(t),\ldots ,A_{n}(t)))=\mathrm{E}%
A(t)^{q}.
\end{align*}%
Thus, we are going to obtain an estimate from below for $\mathrm{E}%
A_{n}(t)^{q}$ for a suitable choice of $n$. Clearly, by the latter
inequality, this estimate will hold for $\mathrm{E}A(t)^{q}$ as well.

We start with a change of variable 
\begin{equation*}
A_{n}(t)=\int_{0}^{t}\Lambda _{n}(s)ds=t\int_{0}^{1}\Lambda _{n}(ut)du\equiv
t\widetilde{A}_{n}(t).
\end{equation*}%
Clearly $\widetilde{A}_{n}(t)$ is a martingale for any fixed $t$.

We are going to treat the cases $q\ge 1$ and $q\le 1$ separately. This is
due to the fact that for $q\ge 1$, the sequences $\widetilde A_n(t)^q$ and $%
A_n(t)^q$ are submartingales while for $q\in(0,1)$ the sequences are
supermartingales with respect to the filtration $\mathcal{F}%
_n=\sigma(\Lambda^{(1)},\ldots,\Lambda^{(n)})$.

We start with an upper bound for $q\geq 1$. Let $n_{t}=-[\log _{b}t]$ be the
biggest integer such that $n_{t}\leq -\log _{b}t$. We use the H\"{o}lder
inequality in the form, 
\begin{equation*}
\left( \int_{0}^{1}|fg|\right) ^{q}=\left(
\int_{0}^{1}|f||g|^{1/q}|g|^{1/p}\right) ^{q}\leq \left(
\int_{0}^{1}|f|^{q}|g|\right) \left( \int_{0}^{1}|g|\right) ^{q/p},
\end{equation*}%
where $1/q+1/p=1$. It follows from the latter inequality, 
\begin{equation*}
\left( \int_{0}^{1}\prod_{k=0}^{n}\Lambda ^{(k)}(ut)du\right) ^{q}\leq
\left( \int_{0}^{1}\left( \prod_{k=0}^{n_{t}-1}\Lambda ^{(k)}(ut)\right)
^{q}\prod_{k=n_{t}}^{n}\Lambda ^{(k)}(ut)du\right) \left(
\int_{0}^{1}\prod_{k=n_{t}}^{n}\Lambda ^{(k)}(ut)du\right) ^{q/p}.
\end{equation*}%
Applying expectation to both sides we obtain, using independence of $\Lambda
^{(k)}$ of each other, 
\begin{equation*}
\mathrm{E}\widetilde{A}_{n}(t)^{q}\leq \left(
\int_{0}^{1}\prod_{k=0}^{n_{t}-1}\mathrm{E}(\Lambda
^{(k)})^{q}(ut)\prod_{k=n_{t}}^{n}\mathrm{E}\left( \Lambda ^{(k)}(ut)\left(
\int_{0}^{1}\prod_{k=n_{t}}^{n}\Lambda ^{(k)}(vt)dv\right) ^{q/p}\right)
du\right) .
\end{equation*}%
By the stationarity of the process $\Lambda (t)$ we have 
\begin{equation*}
\prod_{k=0}^{n_{t}-1}\mathrm{E}(\Lambda ^{(k)})^{q}(ut)=\left( \mathrm{E}%
\Lambda (0)^{q}\right) ^{n_{t}}\leq \left( \mathrm{E}\Lambda (0)^{q}\right)
^{-\log _{b}t}=t^{-\log _{b}\mathrm{E}\Lambda (0)^{q}}.
\end{equation*}%
Therefore, 
\begin{align*}
\mathrm{E}\widetilde{A}_{n}(t)^{q}& \leq t^{-\log _{b}\mathrm{E}\Lambda
(0)^{q}}\mathrm{E}\int_{0}^{1}\prod_{k=n_{t}}^{n}\left( \Lambda
^{(k)}(ut)\left( \int_{0}^{1}\prod_{k=n_{t}}^{n}\Lambda ^{(k)}(vt)dv\right)
^{q/p}\right) du \\
& =t^{-\log _{b}\mathrm{E}\Lambda (0)^{q}}\mathrm{E}\left(
\int_{0}^{1}\prod_{k=n_{t}}^{n}\Lambda ^{(k)}(ut)du\right) ^{1+q/p}=t^{-\log
_{b}\mathrm{E}\Lambda (0)^{q}}\mathrm{E}\left(
\int_{0}^{1}\prod_{k=n_{t}}^{n}\Lambda ^{(k)}(b^{k}ut)du\right) ^{q} \\
& =t^{-\log _{b}\mathrm{E}\Lambda (0)^{q}}\mathrm{E}\left(
\int_{0}^{1}\prod_{k=0}^{n-n_{t}}\Lambda ^{(k)}(b^{k}ub^{-[\log _{b}t]+\log
_{b}t})du\right) ^{q}=t^{-\log _{b}\mathrm{E}\Lambda (0)^{q}}\mathrm{E}%
\widetilde{A}_{n-n_{t}-1}(b^{-[\log _{b}t]+\log _{b}t})^{q}.
\end{align*}%
Now note that 
\begin{equation*}
\mathrm{E}\widetilde{A}_{n-n_{t}-1}(b^{-[\log _{b}t]+\log
_{b}t})^{q}=b^{[log_{b}t]-\log _{b}t}\mathrm{E}A(b^{-[\log _{b}t]+\log
_{b}t})^{q}\leq b\sup_{s\in \lbrack 0,1]}\mathrm{E}A(s)^{q}.
\end{equation*}

This bound is uniform in $n$ and therefore, 
\begin{equation*}
\mathrm{E}A(t)^{q}\leq bt^{q-\log _{b}\mathrm{E}\Lambda (0)^{q}}\sup_{s\in
\lbrack 0,1]}\mathrm{E}A(s)^{q}.
\end{equation*}

Now we turn to the lower bound for $q\ge 1$.

Since $\widetilde{A}_{n}(t)$ is a submartingale, 
\begin{equation*}
\mathrm{E}\widetilde{A}(t)^{q}\geq \mathrm{E}\widetilde{A}_{n_{t}}(t)^{q},
\end{equation*}%
where $n_{t}=[-\log _{b}t]$.

We are going to obtain a recursive estimate for $\mathrm{E}\widetilde{A}%
_{n}(t)$. First, 
\begin{align*}
\mathrm{E}\widetilde{A}_{n+1}(t)^{q}& =\mathrm{E}\left( \int_{0}^{1}\Lambda
_{n}(ut)\Lambda ^{(n+1)}(b^{n+1}ut)du\right) ^{q} \\
& =\mathrm{E}\left( \int_{0}^{1}\Lambda _{n}(ut)(\Lambda
^{(n+1)}(b^{n+1}ut)-\Lambda ^{(n+1)}(0))du+\widetilde{A}_{n}(t)\Lambda
^{(n+1)}(0)\right) ^{q}.
\end{align*}%
Now we can use an elementary estimate of the form: if $a+b>0$ and $b>0$ then 
\begin{equation}
(a+b)^{q}\geq qab^{q-1}+b^{q}  \label{eq_elem}
\end{equation}%
for $q\geq 1$. This estimate is easy to prove by analyzing the function $%
(1+t)^{q}-1-qt$ for $t\geq -1.$ Applying (\ref{eq_elem}) we obtain 
\begin{align}
\mathrm{E}\widetilde{A}_{n+1}(t)^{q}& \geq q\mathrm{E}\left[ \left( 
\widetilde{A}_{n}(t)\Lambda ^{(n+1)}(0)\right) ^{q-1}\int_{0}^{1}\Lambda
_{n}(ut)(\Lambda ^{(n+1)}(b^{n+1}ut)-\Lambda ^{(n+1)}(0))du\right]  \notag
\label{eq_E1E2} \\
& +\mathrm{E}\left( \widetilde{A}_{n}(t)\Lambda ^{(n+1)}(0)\right)
^{q}\equiv E_{1}+E_{2}.
\end{align}%
The second expectation is straightforward, 
\begin{equation}
E_{2}=\mathrm{E}\left( \int_{0}^{1}\Lambda _{n}(ut)\Lambda
^{(n+1)}(0)du\right) ^{q}=\mathrm{E}\Lambda (0)^{q}\mathrm{E}\widetilde{A}%
_{n}(t)^{q},  \label{eq_E2}
\end{equation}%
where we use independence of $\Lambda _{n}$ and $\Lambda ^{(n+1)}$. For the
first expectation, rearranging the terms, we have 
\begin{align*}
E_{1}& =q\mathrm{E}\left[ \int_{0}^{1}\widetilde{A}_{n}(t)^{q-1}\Lambda
_{n}(ut)(\Lambda ^{(n+1)}(0))^{q-1}(\Lambda ^{(n+1)}(b^{n+1}ut)-\Lambda
^{(n+1)}(0))du\right] \\
& =q\int_{0}^{1}\mathrm{E}\widetilde{A}_{n}(t)^{q-1}\Lambda _{n}(ut)\mathrm{E%
}(\Lambda ^{(n+1)}(0))^{q-1}(\Lambda ^{(n+1)}(b^{n+1}ut)-\Lambda
^{(n+1)}(0))du.
\end{align*}%
By the definition of $\rho _{q}$, see \eqref{eq_q1}, for all $u\in \lbrack
0,1],$ 
\begin{equation*}
\mathrm{E}(\Lambda ^{(n+1)}(0))^{q-1}(\Lambda ^{(n+1)}(b^{n+1}ut)-\Lambda
^{(n+1)}(0))\geq \mathrm{E}\Lambda (0)^{q}\rho _{q}(b^{n+1}t).
\end{equation*}%
Therefore, 
\begin{align*}
E_{1}& \geq q\int_{0}^{1}\mathrm{E}\widetilde{A}_{n}(t)^{q-1}\Lambda
_{n}(ut)du\mathrm{E}\Lambda (0)^{q}\rho _{q}(b^{n+1}t) \\
& \geq q\mathrm{E}\left[ \widetilde{A}_{n}(t)^{q-1}\int_{0}^{1}\Lambda
_{n}(ut)du\right] \mathrm{E}\Lambda (0)^{q}\rho _{q}(b^{n+1}t) \\
& =q\mathrm{E}\widetilde{A}_{n}(t)^{q}\mathrm{E}\Lambda (0)^{q}\rho
_{q}(b^{n+1}t)
\end{align*}%
Therefore 
\begin{equation*}
E_{1}\geq q\mathrm{E}\widetilde{A}_{n}(t)^{q}\mathrm{E}\Lambda (0)^{q}\rho
_{q}(b^{n-n_{t}}).
\end{equation*}%
The latter inequality together with (\ref{eq_E1E2}) and (\ref{eq_E2}) gives
us 
\begin{equation}
\mathrm{E}\widetilde{A}_{n+1}(t)^{q}\geq \mathrm{E}\widetilde{A}_{n}(t)^{q}%
\mathrm{E}\Lambda (0)^{q}\left( 1+q\rho _{q}(b^{n-n_{t}})\right)
\label{eq_ind_step}
\end{equation}%
Now we can iterate it. First fix $N^{\ast }$ such that $|q\rho
_{q}(b^{-n})|<1$ for $n>N^{\ast }$. Then, iterating (\ref{eq_ind_step}), we
obtain 
\begin{equation}  \label{eq.d}
\mathrm{E}\widetilde{A}_{n_{t}-N^{\ast }}^{q}\geq (\mathrm{E}\Lambda
(0)^{q})^{n_{t}-N^{\ast }}\prod_{n=0}^{n_{t}-N\ast }\left( 1+q\rho
_{q}(b^{n-n_{t}})\right) \geq (\mathrm{E}\Lambda (0)^{q})^{n_{t}-N^{\ast
}}\prod_{n=N^{\ast }}^{\infty }\left( 1+q\rho _{q}(b^{-n})\right) .
\end{equation}%
It is sufficient to note that the latter product is strictly positive due to
(\ref{eq:ass_scaling}). As $\widetilde{A}_{n}(t)^{q}$ is a submartingale, we
have $\mathrm{E}\widetilde{A}(t)^{q}\geq \mathrm{E}\widetilde{A}%
_{n_{t}-N^{\ast }}^{q}$ and the required lower bound for $q>1$ follows. One
can also see that $\widetilde A(t)$ is non-degenerate. Indeed, by our
assumptions $\mathrm{E }\Lambda(0)^q>0$ and the infinite product in %
\eqref{eq.d} is strictly positive.

The proof for $q\in (0,1)$ is symmetric. For these values of $q$ and a fixed 
$t$, the process $\widetilde{A}_{n}(t)^{q}$ is a supermartingale with
respect to the natural filtration $\mathcal{F}_{n}=\sigma (\Lambda
^{(1)},\ldots ,\Lambda ^{(n)})$. The bound from below is proved using the
reverse H\"{o}lder inequality for $q\in (0,1)$ and $p$ such that $1/p+1/q=1$:%
\begin{equation*}
\left( \int_{0}^{1}|fg|\right) \geq \left( \int_{0}^{1}|f|^{q}\right)
^{1/q}\left( \int_{0}^{1}|g|^{p}\right) ^{1/p}.
\end{equation*}%
Note that $p$ is negative. We are going to use this inequality in the form, 
\begin{equation*}
\left( \int_{0}^{1}|fg|\right) ^{q}=\left(
\int_{0}^{1}|f||g|^{1/q}|g|^{1/p}\right) ^{q}\geq \left(
\int_{0}^{1}|f|^{q}|g|\right) \left( \int_{0}^{1}|g|\right) ^{q/p},
\end{equation*}%
It follows from the latter inequality, 
\begin{equation*}
\left( \int_{0}^{1}\prod_{k=0}^{n}\Lambda _{k}(ut)du\right) ^{q}\geq \left(
\int_{0}^{1}\left( \prod_{k=0}^{n_{t}}\Lambda _{k}(ut)\right)
^{q}\prod_{k=n_{t}+1}^{n}\Lambda _{k}(ut)du\right) \left(
\int_{0}^{1}\prod_{k=n_{t}+1}^{n}\Lambda _{k}(ut)du\right) ^{q/p}.
\end{equation*}%
The rest of the proof goes exactly as the proof of the upper bound for $q>1$.

To prove the upper bound, we proceed similarly to the proof of the lower
bound for $q>1$. First we establish a recursive estimate. The elementary
inequality (\ref{eq_elem}) still holds (in the opposite direction), for $%
q\in (0,1)$,$(a+b)^{q}\leq qab^{q-1}+b^{q},$ for $a+b>0,b>0$. Repeating step
by step the arguments for $q>1$ we obtain an upper bound 
\begin{equation*}
\mathrm{E}\widetilde{A}_{n+1}(t)^{q}\leq \mathrm{E}\widetilde{A}_{n}(t)^{q}%
\mathrm{E}\Lambda (0)^{q}\left( 1+q\rho _{q}(b^{n-n_{t}})\right) .
\end{equation*}%
Applying this bound recursively 
\begin{equation*}
\mathrm{E}A_{n_{t}-N^{\ast }}^{q}\leq (\mathrm{E}\Lambda
(0)^{q})^{n_{t}-N^{\ast }}\prod_{n=0}^{n_{t}-N\ast }\left( 1+q\rho
_{q}(b^{n-n_{t}})\right) \leq (\mathrm{E}\Lambda (0)^{q})^{n_{t}-N^{\ast
}}\prod_{n=N^{\ast }}^{\infty }\left( 1+q\rho _{q}(b^{-n})\right) .
\end{equation*}%
It is sufficient to note that the latter product converge due to (\ref%
{eq:ass_scaling}). As $\widetilde{A}_{n}(t)^{q}$ is a supermartingale, we
have $\mathrm{E}\widetilde{A}(t)^{q}\leq \mathrm{E}\widetilde{A}%
_{n_{t}-N^{\ast }}^{q}$ and the required upper bound for $q<1$ follows.

\qed

\section{ Log-normal scenario with possible long-range dependence}

The log-normal hypothesis of Kolmogorov \cite{Kol62} features prominently in
turbulent cascades. In this section, we provide a related model, namely the
log-normal scenario, for multifractal products of stochastic processes. \ In
fact, this log-normal scenario has its origin in Kahane \cite{Kah85, Kah87}.
In this section we present a general result on log-normal scenario for a
model with possible long-range dependence.

In this Section we consider a mother process of the form 
\begin{equation}
\Lambda (t)=\exp \left\{ X(t)-\frac{1}{2}\sigma _{X}^{2}\right\} ,
\label{eq:defn.stationary.lognormal}
\end{equation}
where $X(t),$ $t\in \lbrack 0,1]$ is a zero-mean Gaussian, measurable,
separable stochastic process with covariance function 
\begin{equation}
R_{X}(\tau )=\sigma _{X}^{2}\mathrm{Corr}(X(t),X(t+\tau )))
\label{eq:defn.stationary.lognormal.cov}
\end{equation}
We combine Theorems \ref{Th3.1} and \ref{thm:scaling} for this special case
in order to have a precise scaling law for the moments.

For the log-normal process we obtain the following specifications of the
moment generating functions (\ref{eq:mgf1})\ and (\ref{eq:mgf2}): 
\begin{eqnarray*}
M(\zeta ) &=&\mathrm{E}\exp \left\{ \zeta \left( X(t)-\frac{1}{2}\sigma
_{X}^{2}\right) \right\} =e^{\frac{1}{2}\sigma _{X}^{2}(\zeta ^{2}-\zeta )},%
\text{\quad }\zeta \in \mathbb{R}^{1}, \\
M(\zeta _{1},\zeta _{2};t_{1}-t_{2}) &=&\mathrm{E}\exp \left\{ \zeta
_{1}\left( X(t_{1})-\frac{1}{2}\sigma _{X}^{2}\right) +\zeta _{2}\left(
X(t_{2})-\frac{1}{2}\sigma _{X}^{2}\right) \right\}  \notag \\
&=&\exp \left\{ \frac{1}{2}\sigma _{X}^{2}\left[ \zeta _{1}^{2}-\zeta
_{1}+\zeta _{2}^{2}-\zeta _{2}\right] +\zeta _{1}\zeta
_{2}R_{X}(t_{1}-t_{2})\right\} ,\text{\quad }\zeta _{1},\zeta _{2}\in 
\mathbb{R}^{1},
\end{eqnarray*}%
where $\sigma _{X}^{2}\in (0,\infty ).$ It turns out that, in this case, 
\begin{eqnarray*}
M(1) &=&1;\text{\quad }M(2)=e^{\sigma _{X}^{2}};\text{ \quad }\sigma
_{\Lambda }^{2}=e^{\sigma _{X}^{2}}-1; \\
\text{\textrm{Cov}}(\Lambda (t_{1}),\Lambda (t_{2}))
&=&M(1,1;t_{1}-t_{2})-1=e^{R_{X}(t_{1}-t_{2})}-1
\end{eqnarray*}%
and%
\begin{equation*}
\log _{b}\mathrm{E}\Lambda (t)^{q}=\frac{(q^{2}-q)\sigma _{X}^{2}}{2\log b}%
,\quad q>0.
\end{equation*}

Note that 
\begin{equation*}
e^{R_{X}(t_{1}-t_{2})}-1\geq R_{X}(t_{1}-t_{2}).
\end{equation*}

Using Theorem~\ref{Th3.1} and Theorem~\ref{thm:scaling} we obtain

\begin{theorem}
\label{Th5.1} Let $X(t)$ be a zero-mean Gaussian measurable separable
stochastic process with the correlation function 
\begin{equation}
\mathrm{Corr}(X(t),X(t+\tau ))\leq C\tau ^{-\alpha },\text{\quad }\alpha >0,
\label{3.3}
\end{equation}%
for sufficiently large $\tau $, and for some $a>0$, 
\begin{equation}
1-\mathrm{Corr}(X(t),X(t+\tau ))\leq C\left\vert \tau \right\vert ^{a},
\label{cond0}
\end{equation}%
for sufficiently small $\tau $. Assume that 
\begin{equation}
b>\exp \left\{ q^{\ast }\sigma _{X}^{2}/2\right\} ,  \label{cond1}
\end{equation}%
where $q^{\ast }\geq 2$ is a fixed integer. Then the stochastic processes 
\begin{equation*}
A_{n}(t)=\int_{0}^{t}\prod_{j=0}^{n}\Lambda ^{(j)}\left( sb^{j}\right)
ds,t\in \lbrack 0,1]
\end{equation*}%
converge in $\mathcal{L}_{q},0<q\leq q^{\ast }$ to the stochastic process $%
A(t),t\in \lbrack 0,1],$ as $n\rightarrow \infty ,$ such that 
\begin{equation}
\mathrm{E}A(t)^{q}\sim t^{\varsigma (q)},q\in \lbrack 0,q^{\ast }],
\label{lognormal_scaling}
\end{equation}%
and the scaling function is given by 
\begin{equation*}
\varsigma (q)=-aq^{2}+(a+1)q,q\in \lbrack 0,q^{\ast }],
\end{equation*}%
where 
\begin{equation*}
a=\frac{\sigma _{X}^{2}}{2\log b}.
\end{equation*}%
Moreover, if 
\begin{equation*}
\mathrm{Corr}(X(t),X(t+\tau ))=\frac{L(\tau )}{\left\vert \tau \right\vert
^{\alpha }},\alpha >0,
\end{equation*}%
where $L$ is a slowly varying at infinity function, bounded on every bounded
interval, then 
\begin{equation}
\text{\textrm{Var}}A(t)\geqslant t^{2-\alpha }\sigma
_{X}^{2}\int_{0}^{1}\int_{0}^{1}\frac{L(t\left\vert u-v\right\vert )dudw}{%
L(t)\left\vert u-w\right\vert ^{\alpha }},0<\alpha <1,  \label{3.4}
\end{equation}%
and 
\begin{equation}
\text{\textrm{Var}}A(t)\geqslant 2t\sigma _{X}^{2}\int_{0}^{t}(1-\frac{u}{t})%
\frac{L(u)}{\left\vert u\right\vert ^{\alpha }}du,\alpha \geq 1.
\label{3.41}
\end{equation}
\end{theorem}

\begin{remark}
We interpret the inequality (\ref{3.4}) as a form of long-range dependence
of the limiting process.
\end{remark}

\begin{remark}
Note that the correlation function $\mathrm{Corr}(X(t),X(t+\tau
))=(1+\left\vert \tau \right\vert ^{2})^{-\alpha /2},$ $\alpha >0,$
satisfies all assumptions of the Theorem 2 (with $L(\tau )=\left\vert \tau
\right\vert ^{\alpha }/(1+\left\vert \tau \right\vert ^{2})^{\alpha /2}$),
among the others.
\end{remark}

\begin{proof}
We will prove $\mathcal{L}_{q^\ast}$ convergence by applying Theorem~\ref%
{Th3.1}, where $q^\ast\ge 2$ is an integer. Hence $\mathcal{L}_{q}$
convergence will hold forn any $q\ge q^\ast$. To simplify notation we will
write $q$ instead of $q^\ast$ when proving $\mathcal{L}_{q^\ast}$
convergence.

The moment generating function of the multidimensional normal distribution
is given by the following expression 
\begin{equation*}
M(\zeta _{1},\zeta _{2},\ldots ,\zeta _{q})=\mathrm{E}e^{\zeta
_{1}X(s_{1})+\cdots +\zeta _{p}X(s_{q})}=\exp \left\{\frac{1}{2}
\sum_{i=1}^{q}\sum_{j=1}^{q}\zeta_i\zeta_jR_{X}(|s_{i}-s_{j}|)\right\}.
\end{equation*}

One can immediately see that%
\begin{equation*}
\mathrm{E}\left( \Lambda (s_{1})\Lambda (s_{2})\ldots \Lambda (s_{q})\right)
=\mathrm{E}e^{X(s_{1})-\frac{1}{2}\sigma _{X}^{2}}\ldots e^{X(s_{q})-\frac{1%
}{2}\sigma _{X}^{2}}
\end{equation*}%
\begin{equation*}
=M(1,1,\ldots ,1)e^{-\frac{q}{2}\sigma _{X}^{2}}=e^{\frac{1}{2}%
\sum_{i=1}^{q}\sum_{j=1}^{q}R_{X}(|s_{i}-s_{j}|)}=e^{-\frac{q}{2}\sigma
_{X}^{2}=e^{\sum_{1\leq i<j\leq q}R_{X}(s_{j}-s_{i})}}.
\end{equation*}
We can now substitute this into (~\ref{RO}) and obtain 
\begin{equation*}
\rho (u_{1},u_{2},\ldots ,u_{q-1})=\exp \left\{ \sum_{1\leq i<j\leq
q-1}R_{X}(u_{i}+\cdots +u_{j})\right\} .
\end{equation*}%
Since the function $R_{X}(u)$ is monotone decreasing in $u$, function $\rho
(u_{1},\ldots ,u_{q-1})$ is monotone decreasing in all arguments. Next we
need to check the mixing condition (\ref{eq:mixing}). Let $1\leq
i_{1}<i_{2}<\ldots \leq i_{m}$ and $u_{i}=A$ if $i\in \{i_{1},\ldots
,i_{m}\} $ and $0$ otherwise. Then, as $A\rightarrow \infty $, and $%
i_{0}=0,i_{m+1}=q$ 
\begin{multline}
\lim_{A\rightarrow \infty }\rho (u_{1},\ldots ,u_{q-1})=\exp \left\{
\sum_{1\leq k\leq m+1}\sum_{i_{k-1}<i<j<i_{k}}R_{X}(u_{i}+\cdots
+u_{j})\right\}  \label{eq:lognormal.mixing} \\
=\mathrm{E}\Lambda (0)^{i_{1}}\mathrm{E}\Lambda (0)^{i_{2}-i_{1}}\cdot
\ldots \mathrm{E}\Lambda (0)^{q-i_{m}},
\end{multline}%
where we used that $\mathrm{E}\Lambda (0)^{l}=e^{\frac{l(l-1)}{2}\sigma
_{X}^{2}}.$ Finally, we should check the convergence of the series (\ref%
{eq:sum_rho_finite}). We have, 
\begin{equation*}
\exp \{R_{X}(qb^{n}\}\leq \rho (b^{n},\ldots ,b^{n})\leq \exp \left\{ \frac{%
q(q-1)}{2}R_{X}(b^{n}\right\} .
\end{equation*}%
As $n\rightarrow \infty $, $R_{X}(b^{n})\rightarrow 0$. Hence 
\begin{equation*}
(1+o(1))R_{X}(qb^{n})\leq \rho (b^{n},\ldots ,b^{n})-1\leq (1+o(1))\frac{%
q(q-1)}{2}R_{X}(b^{n}).
\end{equation*}%
As both sums 
\begin{equation*}
\sum_{n=1}^{\infty }R_{X}(qb^{n})<\infty ,\quad \sum_{n=1}^{\infty
}R_{X}(b^{n})<\infty ,
\end{equation*}%
the convergence of the series (\ref{eq:sum_rho_finite}) follows. Condition (%
\ref{eq:b.greater}) becomes 
\begin{equation*}
b^{q-1}>\mathrm{E}\Lambda (0)^{q}=\exp \left\{ \frac{q(q-1)}{2}\sigma
_{X}^{2}\right\} ,
\end{equation*}%
which is equivalent to (\ref{cond1}).

Next we are going to prove scaling (\ref{lognormal_scaling}). For that we
apply the results of Section \ref{section_scaling}. We now do not assume
that $q$ is an integer. We need to show that (\ref{eq:ass_scaling}) holds
for $\rho _{q},$ where $q\in (0,q^{\star })$ and $\rho _{q}$ is defined in (%
\ref{eq_q1}) and (\ref{eq_q2}). For $q>1$ we have, for sufficiently small $s$%
, 
\begin{align*}
|\rho _{q}(s)|& =-\inf_{u\leq 1}\left( \frac{\mathrm{E}\Lambda
(0)^{q-1}\Lambda (su)}{\mathrm{E}\Lambda (0)^{q}}-1\right) =-\inf_{u\leq
1}\left( e^{\sigma _{X}^{2}((q-1)\rho _{X}(su)+1-q)}-1\right) \\
& \leq \sup_{u\leq 1}\left( 1-e^{(1-q)\sigma _{X}^{2}(su)^{a}}\right) \leq
1-e^{(1-q)\sigma _{X}^{2}(s)^{a}}\leq (q-1)\sigma _{X}^{2}s^{a}.
\end{align*}%
Thus using condition (\ref{cond0}) one can immediately see that the series (%
\ref{eq:ass_scaling}) converges. For $q<1$, the same arguments give the
bound 
\begin{equation*}
\rho _{q}(s)\leq (1-q)\sigma _{X}^{2}s^{a}.
\end{equation*}%
Using condition (\ref{cond0}) one can immediately see that the series (\ref%
{eq:ass_scaling}) converges. Therefore, by the results of Section \ref%
{section_scaling} scaling (\ref{lognormal_scaling}) holds. \bigskip
\end{proof}

\section{Geometric Ornstein-Uhlenbeck processes}

This section reviews a number of known results on L\'{e}vy processes (see
Skorokhod 1991, Bertoin 1996, Kyprianou 2006) and OU type processes (see
Barndorff-Nielsen 1998, 2001, Barndorff-Nielsen and Shephard 2001) \ The
geometric OU type processes have been studied also by Matsui and Shieh
(2009).

As standard notation we will write 
\begin{equation*}
\kappa (z)=C\left\{ z;X\right\} =\log \mathrm{E}\exp \left\{ izX\right\}
,\quad z\in \mathbb{R}
\end{equation*}%
for the cumulant function of a random variable $X$, and 
\begin{equation*}
K\left\{ \zeta ;X\right\} =\log \mathrm{E}\exp \left\{ \zeta X\right\} ,%
\text{ \ \ \ }\zeta \in D\subseteq \mathbb{C}
\end{equation*}%
for the L\'{e}vy exponent or Laplace transform or cumulant generating
function of the random variable $X.$ Its domain $D$ includes the imaginary
axis and frequently larger areas.

A random variable $X$ is infinitely divisible if its cumulant function has
the L\'{e}vy-Khintchine form 
\begin{equation}
C\left\{ z;X\right\} =iaz-\frac{d}{2}z^{2}+\int_{\mathbb{R}}\left(
e^{izu}-1-izu\mathbf{1}_{\left[ -1,1\right] }\left( u\right) \right) \nu
\left( du\right) ,  \label{3.1}
\end{equation}%
where $a\in \mathbb{R},$ $d\geq 0$ and $\nu $ is the L\'{e}vy measure, that
is, a non-negative measure on $\mathbb{R}$ such that 
\begin{equation}
\nu \left( \left\{ 0\right\} \right) =0,\text{ \ \ \ \ }\int_{\mathbb{R}%
}\min \left( 1,u^{2}\right) \nu \left( du\right) <\infty .
\label{eq:Levy-Khinchine.exponent.cond}
\end{equation}%
The triplet $(a,d,\nu )$ uniquely determines the random variable $X.$ For a
Gaussian random variable $X\thicksim N\left( a,d\right) ,$ the L\'{e}vy
triplet takes the form $\left( a,d,0\right) .$

A random variable $X$ is self-decomposable if, for all $c\in \left(
0,1\right) ,$ the characteristic function $f\left( z\right) $ of $X$ can be
factorized as $f\left( z\right) =f\left( cz\right) f_{c}\left( z\right) $
for some characteristic function $f_{c}\left( z\right) ,$ $z\in \mathbb{R}.$
A homogeneous L\'{e}vy process $Z=\{Z\left( t\right) ,$ $t\geq 0\}$ is a
continuous (in probability), c\`{a}dl\`{a}g process with independent and
stationary increments and $Z\left( 0\right) =0$ (recalling that a c\`{a}dl%
\`{a}g process has right-continuous sample paths with existing left limits.)
For such processes we have $C\left\{ z;Z\left( t\right) \right\} =tC\left\{
z;Z\left( 1\right) \right\} $ and $Z\left( 1\right) $ has the L\'{e}%
vy-Khintchine representation (\ref{3.1}).

If $X$ is self-decomposable, then there exists a stationary stochastic
process $\{X\left( t\right) ,$ $t\geq 0\},$ such that $X\left( t\right) 
\overset{d}{=}X\,\,$and 
\begin{equation}
X\left( t\right) =e^{-\lambda t}X\left( 0\right) +\int_{(0,t]}e^{-\lambda
\left( t-s\right) }dZ\left( \lambda s\right) ,  \label{eq:defn.ou.int}
\end{equation}%
for all $\lambda >0$ (see Barndorff-Nielsen 1998). Conversely, if $\left\{
X\left( t\right) ,t\geq 0\right\} $ is a stationary process and $\left\{
Z\left( t\right) ,t\geq 0\right\} $ is a L\'{e}vy process, independent of $%
X\left( 0\right) ,$ such that $X\left( t\right) $ and $Z\left( t\right) $
satisfy the It\^{o} stochastic differential equation 
\begin{equation}
dX\left( t\right) =-\lambda X\left( t\right) dt+dZ\left( \lambda t\right) ,
\label{eq:defn.ou}
\end{equation}%
for all $\lambda >0,$ then $X\left( t\right) $ is self-decomposable. A
stationary process $X\left( t\right) $ of this kind is said to be an OU type
process. The process $Z\left( t\right) $ is termed the background driving L%
\'{e}vy process (BDLP) corresponding to the process $X\left( t\right) .$ In
fact (\ref{eq:defn.ou.int}) is the unique (up to indistinguishability)
strong solution to Eq. (\ref{eq:defn.ou}) (Sato 1999, Section 17). The
meaning of the stochastic integral in (\ref{eq:defn.ou.int}) was detailed in
Applebaum (2009, p. 214).

Let $X\left( t\right) $ be a square integrable OU process. Then $X\left(
t\right) $ has the correlation function 
\begin{equation}
\mathrm{Corr}(X(0),X(t))=r_{X}\left( t\right) =\exp \left\{ -\lambda
\left\vert t\right\vert \right\} .  \label{3.5}
\end{equation}%
The cumulant transforms of $X=X(t)$ and $Z\left( 1\right) $ are related by 
\begin{equation*}
C\left\{ z;X\right\} =\int_{0}^{\infty }C\left\{ e^{-s}z;Z\left( 1\right)
\right\} ds=\int_{0}^{{}z}C\left\{ \xi ;Z\left( 1\right) \right\} \frac{d\xi 
}{\xi },C\left\{ z;Z\left( 1\right) \right\} =z\frac{\partial C\left\{
z;X\right\} }{\partial z}.
\end{equation*}%
Suppose that the L\'{e}vy measure $\nu $ of $X$ has a density function $%
p\left( u\right) ,u\in \mathbb{R},$ which is differentiable. Then the L\'{e}%
vy measure $\tilde{\nu}$ of $Z\left( 1\right) $ has a density function $%
q\left( u\right) ,u\in \mathbb{R}$, and $p\,$and $q$ are related by 
\begin{equation}
q\left( u\right) =-p\left( u\right) -up^{\prime }\left( u\right)  \label{3.6}
\end{equation}%
(see Barndorff-Nielsen 1998).

The logarithm of the characteristic function of a random vector $\left(
X(t_{1}),...,X(t_{m})\right) $ is of the form%
\begin{equation}
\log \mathrm{E\exp }\left\{ i(z_{1}X(t_{1})+...+z_{m}X(t_{m})\right\} =\int_{%
\mathbb{R}}\kappa (\sum_{j=1}^{m}z_{j}e^{-\lambda (t_{j}-s)}\mathbf{1}
_{[0,\infty )}(t_{j}-s))ds,  \label{3.7}
\end{equation}%
where%
\begin{equation*}
\kappa (z)=\log \mathrm{E\exp }\left\{ izZ(1)\right\} =C\left\{
z;Z(1)\right\} ,
\end{equation*}%
and the function (\ref{3.7}) has the form (\ref{3.1}) with L\'{e}vy triplet $%
(\tilde{a},$ $\tilde{d},$ $\tilde{\nu})$ of $Z(1).$

The logarithms of the moment generation functions (if they exist) take the
forms%
\begin{equation*}
\log \mathrm{E\exp }\left\{ \zeta X(t)\right\} =\zeta a+\frac{d}{2}\zeta
^{2}+\int_{\mathbb{R}}(e^{\zeta u}-1-\zeta u\mathbf{1}_{\left[ -1,1\right]
}\left( u\right) )\nu \left( du\right) ,
\end{equation*}%
where $(a,d,\nu )$ is the L\'{e}vy triplet of $X(0)$, or in terms of the L%
\'{e}vy triplet $(\tilde{a},$ $\tilde{d},$ $\tilde{\nu})$ of $Z(1)$%
\begin{align}
& \log \mathrm{E\exp }\left\{ \zeta X(t)\right\} =\tilde{a}\int_{\mathbb{R}%
}(\zeta e^{-\lambda (t-s)}\mathbf{1}_{[0,\infty )}(t-s))ds+\frac{\tilde{d}}{2%
}\zeta ^{2}\int_{\mathbb{R}}(\zeta e^{-\lambda (t-s)}\mathbf{1}_{[0,\infty
)}(t-s))^{2}ds  \notag \\
& +\int_{\mathbb{R}}\int_{\mathbb{R}}[\exp \left\{ u\zeta e^{-\lambda (t-s)}%
\mathbf{1}_{[0,\infty )}(t-s)\right\} -1-u\left( \zeta e^{-\lambda (t-s)}%
\mathbf{1}_{[0,\infty )}(t-s)\right) \mathbf{1}_{\left[ -1,1\right] }\left(
u\right) ]\tilde{\nu}\left( du\right) ds,  \label{3.8}
\end{align}%
and%
\begin{align}
& \log \mathrm{E}\exp \{\zeta _{1}X(t_{1})+\zeta _{2}X(t_{2})\}  \notag \\
& =\tilde{a}\int_{\mathbb{R}}\left( \sum_{j=1}^{2}\zeta _{j}e^{-\lambda
(t_{j}-s)}\mathbf{1}_{[0,\infty )}(t_{j}-s))ds+\frac{\tilde{d}}{2}\zeta
^{2}\int_{\mathbb{R}}(\sum_{j=1}^{2}\zeta _{j}e^{-\lambda (t_{j}-s)}\mathbf{1%
}_{[0,\infty )}(t_{j}-s))^{2}\right) ds  \notag \\
& \hspace{0.5cm}+\int_{\mathbb{R}}\int_{\mathbb{R}}[\exp \left\{
u\sum_{j=1}^{2}\zeta _{j}e^{-\lambda (t_{j}-s)}\mathbf{1}_{[0,\infty
)}(t_{j}-s)\right\} -1  \notag \\
& \hspace{0.5cm}-u\left( \sum_{j=1}^{2}\zeta _{j}e^{-\lambda (t_{j}-s)}%
\mathbf{1}_{[0,\infty )}(t_{j}-s)\right) \mathbf{1}_{\left[ -1,1\right]
}\left( u\right) ]\tilde{\nu}\left( du\right) ds.  \label{eq39}
\end{align}

Let us consider a geometric OU-type process as the mother process: 
\begin{equation*}
\Lambda (t)=e^{X(t)-c_{X}},c_{X}=\log \mathrm{E}e^{X(0)},M(\zeta )=\mathrm{E}%
e^{\zeta (X(t)-c_{X})}, M_0(\zeta )=\mathrm{E}e^{\zeta X(t)}
\end{equation*}%
where $X(t),t\in \mathbb{R}_{+}$, is the OU-type stationary process (\ref%
{eq:defn.ou.int}). Note that 
\begin{equation*}
\frac{M_{0}(q)}{M_{0}(1)^{q}}=\frac{M(q)}{M(1)^{q}}.
\end{equation*}

Then the correlation function of the mother process is of the form. 
\begin{equation}
\mathrm{Corr}(\Lambda (t),\Lambda (t+\tau ))=\frac{M(1,1;\tau )-1}{M(2)-1},
\label{OU03}
\end{equation}%
where now 
\begin{align}
M(\zeta _{1},\zeta _{2};\tau )&=\mathrm{E}\exp \{\zeta
_{1}(X(t_{1})-c_{X})+\zeta _{2}(X(t_{2})-c_{X})\}  \notag \\
&=\exp \left\{ -(\zeta _{1}+\zeta _{2})c_{X}\right\} \mathrm{E}\exp \{\zeta
_{1}X(t_{1})+\zeta _{2}X(t_{2})\},  \label{OUM}
\end{align}%
and $\mathrm{E}\exp \{\zeta _{1}X(t_{1})+\zeta _{2}X(t_{2})\}$ is defined by
(\ref{eq39}).

To prove that a geometric OU process satisfies the covariance decay
condition (\ref{eq:ass_scaling}) in Theorem \textit{\ref{thm:scaling}}, the
expression given by (\ref{eq39}) is not ready to yield the decay as $%
t_{2}-t_{1}\rightarrow \infty $. 

The following result plays a key role in multifractal analysis of geometric
OU processes.

\begin{theorem}
\label{thm:ou} Let $X(t),t\in \mathbb{R}_{+}$ be an OU-type stationary
process (\ref{eq:defn.ou.int}) such that the L\'{e}vy measure $\nu $ in (\ref%
{3.1}) of the random variable $X(0)$ satisfies the condition: for an integer 
$q^{\ast }\geq 2,$ 
\begin{equation}
\int_{\left\vert x\right\vert \geq 1}xe^{q^{\ast }x}\nu (dx)<\infty .
\label{eq:existence.derivative}
\end{equation}%
Then, for any fixed $b$ such that 
\begin{equation}
b>\left\{ \frac{M_{0}(q^{\ast })}{M_{0}(1)^{q^{\ast }}}\right\} ^{\frac{1}{%
q^{\ast }-1}},  \label{eq:ou.cond.b}
\end{equation}%
the sequence of stochastic processes 
\begin{equation*}
A_{n}(t)=\int_{0}^{t}\prod_{j=0}^{n}\Lambda ^{(j)}\left( sb^{j}\right)
ds,t\in \lbrack 0,1]
\end{equation*}%
converges in $\mathcal{L}_{q}$ to the stochastic process $A(t)\in \mathcal{L}%
_{q},$ as $n\rightarrow \infty $, for every fixed $t\in \lbrack 0,1]$. The
limiting process $A(t),t\in \lbrack 0,1]$ satisfies 
\begin{equation*}
\mathrm{E}A^{q}(t)\sim t^{q-\log _{b}\mathrm{E}\Lambda ^{q}\left( t\right)
},\quad q\in \lbrack 0,q^{\ast }].
\end{equation*}%
The scaling function is given by 
\begin{equation}
\varsigma (q)=q-\log _{b}\mathrm{E}\Lambda ^{q}\left( t\right) =q\left( 1+%
\frac{c_{X}}{\log b}\right) -\log _{b}M_{0}(q),\quad q\in \lbrack 0,q^{\ast
}].  \label{R1}
\end{equation}%
In addition, 
\begin{equation}
\mathrm{Var}A(t)\geqslant 2t\int_{0}^{t}\left( 1-\frac{s}{t}\right)
(M(1,1;s)-1)ds,  \label{R2}
\end{equation}%
where the bivariate moment generating function $M(\zeta _{1},\zeta
_{2};t_{1}-t_{2})$ is given by (\ref{OUM})
\end{theorem}

\textit{Proof of Theorem~\ref{thm:ou}} We are starting with $\mathcal{L}_{q}$
convergence. To show the convergence we apply Theorem~\ref{Th3.1}. It is
sufficient to show the convergence for $q=q^{\ast }$ since the convergence
for $q<q^{\ast }$ immediately follows from the convergence for $q=q^{\ast }$%
. First we will derive a suitable explicit expression for $\rho
(u_{1},\ldots ,u_{q-1})$. Put $s_{1}=0\leq s_{2}=u_{1}\leq
s_{2}=u_{1}+u_{2},\ldots ,s_{q}=u_{1}+\cdots +u_{q-1}$. Then, 
\begin{equation*}
\rho (u_{1},\ldots ,u_{q-1})=\mathrm{E}\Lambda (s_{1})\ldots \Lambda (s_{q})=%
\mathrm{E}\exp \left\{ X(s_{1})+\ldots +X(s_{q})-qc_{X}\right\} .
\end{equation*}%
Using representation~(\ref{eq:defn.ou.int}) one can obtain 
\begin{equation*}
X(s_{q})=e^{-\lambda
(s_{q}-s_{q-1})}X(s_{q-1})+\int_{(s_{q-1},s_{q}]}e^{-\lambda
(s_{q}-s)}dZ(\lambda s).
\end{equation*}%
Then,using independence of $X(s_{q-1})$ and the integral $%
\int_{(s_{q-1},s_{q}]}e^{-\lambda (s_{q}-s)}dZ(\lambda s)$ we obtain 
\begin{align*}
& \mathrm{E}\exp \left\{ X(s_{1})+\ldots +X(s_{q})\right\} \\
& \hspace{2cm}=\mathrm{E}\exp \left\{ X(s_{1})+\ldots +(1+e^{-\lambda
(s_{q}-s_{q-1})})X(s_{q-1})\right\} \mathrm{E}e^{\int_{(s_{q-1},s_{q}]}e^{-%
\lambda (s_{q}-s)}dZ(\lambda s)} \\
& \hspace{2cm}=\mathrm{E}\exp \left\{ X(s_{1})+\ldots +(1+e^{-\lambda
(s_{q}-s_{q-1})})X(s_{q-1})\right\} \frac{\mathrm{E}e^{X(s_{q})}}{\mathrm{E}%
e^{e^{-\lambda (s_{q}-s_{q-1})}X(0)}} \\
& \hspace{2cm}=\mathrm{E}\exp \left\{ X(s_{1})+\ldots +(1+e^{-\lambda
(s_{q}-s_{q-1})})X(s_{q-1})\right\} \frac{M_{0}(1)}{M_{0}(e^{-\lambda
(s_{q}-s_{q-1})})}.
\end{align*}

Proceeding further by induction we obtain 
\begin{align*}
& \mathrm{E}\exp \left\{ X(s_{1})+\ldots +X(s_{q})\right\} \\
& =M_{0}(1)\frac{M_{0}(1+e^{-\lambda u_{q-1}})M_{0}(1+e^{-\lambda
u_{q-2}}+e^{-\lambda (u_{q-1}+u_{q-2})})\ldots M_{0}(1+e^{-\lambda
u_{1}}+\ldots +e^{-\lambda (u_{1}+\ldots +u_{q-1})})}{M_{0}(e^{-\lambda
u_{q-1}})M_{0}(e^{-\lambda u_{q-2}}+e^{-\lambda (u_{q-1}+u_{q-2}})\ldots
M_{0}(e^{-\lambda u_{1}}+\ldots +e^{-\lambda (u_{1}+\ldots +u_{q-1})})}.
\end{align*}

Hence 
\begin{multline}
\rho (u_{1},\ldots ,u_{q-1})=\frac{M_{0}(1+e^{-\lambda u_{q-1}})}{%
M_{0}(1)M_{0}(e^{-\lambda u_{q-1}})}\frac{M_{0}(1+e^{-\lambda
u_{q-2}}+e^{-\lambda (u_{q-1}+u_{q-2})})}{M_{0}(1)M_{0}(e^{-\lambda
u_{q-2}}+e^{-\lambda (u_{q-1}+u_{q-2})})}\ldots \times
\label{eq:rho.explicit.ou} \\
\times \frac{M_{0}(1+e^{-\lambda u_{1}+\ldots +e^{-\lambda (u_{1}+\ldots
+u_{q-1})}})}{M_{0}(1)M_{0}(e^{-\lambda u_{1}+\ldots +e^{-\lambda
(u_{1}+\ldots +u_{q-1})}})}.
\end{multline}%
This representation allows us to show monotonicity of $\rho (u_{1},\ldots
,u_{q-1}).$ For that we use the following inequality 
\begin{equation}
\frac{M_{0}(1+s)}{M_{0}(s)}\leq \frac{M_{0}(1+t)}{M_{0}(t)}  \label{eq_kar}
\end{equation}%
for $s\leq t$. This inequality follows from the fact that $\ln M_{0}(t)$ is
a convex function and the Karamata majorisation inequality. Hence 
\begin{equation*}
\frac{M_{0}(1+s)}{M_{0}(1)M_{0}(s)}
\end{equation*}%
is monotone increasing in $s$. Since $e^{-\lambda u}$ is monotone decreasing
in $u$ the representation (\ref{eq:rho.explicit.ou}) implies that $\rho
(u_{1},\ldots ,\ldots ,u_{q-1})$ is monotone decreasing in all variables.

Condition (\ref{eq:b.greater}) becomes 
\begin{equation*}
b^{q-1}>\mathrm{E}\Lambda (0)^q=\frac{M_{0}(q)}{M_{0}(1)^{q}},
\end{equation*}%
which is equivalent to (\ref{eq:ou.cond.b}).

To show the finiteness of the series (\ref{eq:sum_rho_finite}) we are going
to use the following statement.

\begin{lemma}
\label{lem_est} For $s\in \lbrack 0,1]$, the following estimate holds 
\begin{equation}
\frac{M_{0}(1+s)}{M_{0}(1)M(s)}\leq \left( \frac{M_{0}(2)}{M_{0}(1)e^{%
\mathrm{E}X(1)}}\right) ^{s}.  \label{eq309}
\end{equation}
\end{lemma}

\textit{Proof of Lemma~\ref{lem_est} } Function $\ln M_{0}(t)$ is convex.
Therefore, 
\begin{equation*}
\ln M_{0}(1+s)=\ln M_{0}((1-s)+2s)\leq (1-s)\ln M_{0}(1)+s\ln M_{0}(2).
\end{equation*}%
In addition, by the Jensen inequality, 
\begin{equation*}
M_{0}(s)=\mathrm{E}e^{sX(1)}\geq e^{s\mathrm{E}X(1)}.
\end{equation*}%
Together these inequalities imply, 
\begin{equation*}
\frac{M_{0}(1+s)}{M_{0}(1)M_{0}(s)}\leq \frac{M_{0}(1)^{1-s}M_{0}(2)^{s}}{%
M_{0}(1)e^{s\mathrm{E}X(1)}}=\left( \frac{M_{0}(2)}{M_{0}(1)e^{\mathrm{E}%
X(1)}}\right) ^{s}.
\end{equation*}

\qed

Now, using (\ref{eq:rho.explicit.ou}) and monotone decrease of $%
M_{0}(1+s)/M_{0}(s)$ 
\begin{align}  \label{eq310}
1&\leq \rho (b^{n},\ldots ,b^{n})\leq \left( \frac{M_{0}(1+e^{-\lambda
b^{n}})}{M_{0}(1)M_{0}(e^{-\lambda b^{n}})}\right) ^{q} \\
&\leq C^{qe^{-\lambda b^{n}}}\leq 1+o(1)\ln Cqe^{-\lambda b^{n}},  \notag
\end{align}%
where the former inequality follows from Lemma~\ref{lem_est} with $%
C=M_{0}(2)/(M_{0}(1)e^{\mathrm{E}X(1)}).$ Then, convergence of the series (%
\ref{eq:rho.explicit.ou}) follows from the finiteness of the series $%
\sum_{n=1}^{\infty }e^{-\lambda b^{n}} $.

Finally we need to check the mixing condition (\ref{eq:mixing}). Let $1\leq
i_{1}<i_{2}<\ldots \leq i_{m}$ and $u_{i}=A$ if $i\in \{i_{1},\ldots
,i_{m}\} $ and $0$ otherwise. In this context it is convenient to use (\ref%
{eq:rho.explicit.ou}) in the form 
\begin{equation*}
\rho (u_{1},\ldots ,u_{q-1})=\prod_{j=1}^{q-1}\frac{M_{0}(1+%
\sum_{k=j}^{q-1}e^{-\lambda \sum_{l=j}^{k}u_{l}})}{M_{0}(1)M_{0}(%
\sum_{k=j}^{q-1}e^{-\lambda \sum_{l=j}^{k}u_{l}})}
\end{equation*}%
Then, as $A\rightarrow \infty $, and $i_{0}=0,i_{m+1}=q$ 
\begin{multline}
\lim_{A\rightarrow \infty }\rho (u_{1},\ldots ,u_{q-1})=\prod_{\alpha
=1}^{m}\prod_{j=i_{\alpha }+1}^{i_{\alpha +1}-1}\frac{M_{0}(1+%
\sum_{k=j}^{i_{\alpha +1}-1}1)}{M_{0}(1)M_{0}(\sum_{k=j}^{i_{\alpha +1}-1}1)}
\label{eq:ou.mixing} \\
=\prod_{\alpha =1}^{m}\frac{M_{0}(i_{\alpha +1}-i_{\alpha })}{%
M_{0}(1)^{i_{\alpha +1}-i_{\alpha }}}=\prod_{\alpha =1}^{m}\mathrm{E}\Lambda
(0)^{i_{\alpha +1}-i_{\alpha }}.
\end{multline}%
This proves (\ref{eq:mixing}). Therefore Theorem~\ref{Th3.1} gives $\mathcal{%
L}_{q}$ convergence of $A_{n}(t)$.

To prove the scaling property we are going to use the results of Theorem~\ref%
{thm:scaling}. First using representation (\ref{eq:defn.ou.int}), we have
for any $q$, 
\begin{align*}
\mathrm{E}\Lambda (t)\Lambda (0)^{q-1} &=\mathrm{E}\exp%
\{(q-1)X(0)+X(t)-qc_{X}\} \\
&=\mathrm{E}\exp \{(q-1+e^{-\lambda t})X(0)+\int_{(0,t)}e^{-\lambda
(t-s)}dZ(\lambda s)-qc_{X}\} \\
&=\mathrm{E}\exp \left\{ (q-1+e^{-\lambda t})X(0)-qc_{X}\right\} \frac{%
\mathrm{E}\exp \{e^{-\lambda t}X(0)+\int_{(0,t)}e^{-\lambda (t-s)}dZ(\lambda
s)\}}{\mathrm{E}e^{e^{-\lambda t}X(0)}} \\
&=\mathrm{E}\exp \{(q-1+e^{-\lambda t})X(0)-qc_{X}\}\frac{\mathrm{E}\exp
\{X(t)\}}{\mathrm{E}e^{e^{-\lambda t}X(0)}} \\
&=\frac{M_{0}(q-1+e^{-\lambda t})}{M_{0}(1)^{q-1}M(e^{-\lambda t})}.
\end{align*}
Then, 
\begin{equation*}
\frac{\mathrm{E}\Lambda (t)^{q-1}\Lambda (0)^{q-1}}{\mathrm{E}\Lambda (0)^{q}%
}=\frac{M_{0}(q-1+e^{-\lambda t})M_{0}(1)}{M_{0}(q)M_{0}(e^{-\lambda t})}.
\end{equation*}%
For $q>1$ the latter function is monotone decreasing in $t$, as follows from
the Karamata motorization inequality. Hence, 
\begin{equation}
|\rho _{q}(s)|=\sup_{u\in \lbrack 0,1]}\left( 1-\frac{\mathrm{E}\Lambda
(su)\Lambda (0)^{q-1}}{\mathrm{E}\Lambda (0)^{q}}\right) =1-\frac{%
M_{0}(q-1+e^{-\lambda s})M_{0}(1)}{M_{0}(q)M_{0}(e^{-\lambda s})}.
\label{eq:ou.rhoq1}
\end{equation}%
Function $f(x)=\ln M(x)$ is convex. Condition \eqref{eq:existence.derivative}
ensures that the derivative $f^{\prime }(q)$ exists for $q\leq q^{\ast }$.
Then, for any $x\leq q$, 
\begin{equation*}
f(x)-f(q)\geq (x-q)f^{\prime }(q).
\end{equation*}%
In particular for $x=q-1+e^{-\lambda s}$, 
\begin{equation*}
f(q-1+e^{-\lambda s})-f(q)\geq (-1+e^{-\lambda s})f^{\prime }(q).
\end{equation*}%
In addition, by the Jensen inequality, 
\begin{equation*}
M_{0}(e^{-\lambda s})=\mathrm{E}e^{e^{-\lambda s}X(0)}\leq (\mathrm{E}%
e^{X(0)})^{e^{-\lambda s}}=M_{0}(1)^{e^{-\lambda s}}.
\end{equation*}%
The latter two inequalities give 
\begin{align}
|\rho _{q}(s)| &\leq 1-e^{(-1+e^{-\lambda s})(f^{\prime }(q)-f(1))}
\label{eq:ou.rho.q} \\
&\leq (1-e^{-\lambda s})(f^{\prime }(q)-f(1))\leq \lambda s(f^{\prime
}(q)-f(1)).  \notag
\end{align}%
Then 
\begin{equation*}
0\leq \sum_{n=1}^{\infty }|\rho _{q}(b^{-n})|\leq \lambda (f^{\prime
}(q)-f(1))\sum_{n=1}^{\infty }b^{-n}<\infty .
\end{equation*}%
Since we have already shown that $A(t)\in \mathcal{L}_{q}$ for $q<q^*$, we
can apply Theorem~\ref{thm:scaling}.

\qed  

\section{Superpositions of geometric Ornestein-Uhlenbeck processes}

\label{sec:superposition-geom-ornst}

The correlation structures found in applications may be more complex than
the exponential decreasing autocorrelation of the form (\ref{3.5}).
Barndorff-Nielsen (1998) (see also Barndorff-Nielsen and Sheppard (2001))
proposed to consider the following class of autocovariance functions:%
\begin{equation}
R_{m}(t)=\sum_{j=1}^{m}\sigma _{j}^{2}\exp \left\{ -\lambda _{j}\left\vert
t\right\vert \right\} ,  \label{3.11}
\end{equation}%
which is flexible and can be fitted to many autocovariance functions arising
in applications. The role of an integer $m\geq 1$ is discussed in
Barndorff-Nielsen and Sheppard (2001).

In order to obtain models with dependence structure (\ref{3.11}) and given
marginal density with finite variance, we consider stochastic processes
defined by%
\begin{equation*}
dX_{j}\left( t\right) =-\lambda _{j}X_{j}\left( t\right) dt+dZ_{j}\left(
\lambda _{j}t\right) ,~j=1,2,...,m,...
\end{equation*}%
and their finite superposition%
\begin{equation}
X_{m\sup }(t)=X_{1}(t)+...+X_{m}(t),~t\geq 0,  \label{3.12}
\end{equation}%
where $Z_{j},~j=1,2,...,m,...$ are mutually independent L\'{e}vy processes.
Then the solution $X_{j}=\left\{ X_{j}(t),t\geq 0\right\} ,~j=1,2,...,m,$ is
a stationary process. Its correlation function is of the exponential form
(assuming finite variance of the components).

The superposition (\ref{3.12}) has its marginal density given by that of the
random variable 
\begin{equation}
X_{m\sup }(0)=X_{1}(0)+...+X_{m}(0),  \label{3.13}
\end{equation}%
and autocovariance function (\ref{3.11}). One can generalize Theorem~\ref%
{thm:ou} to the case of finite superposition process (\ref{3.12}).

We are interested in the case when the distribution of (\ref{3.13}) is
tractable, for instance when $X_{m}(0)$ belongs to the same class as $%
X_{j}(0),j=1,...,m$ (see the examples in Sections 8--13 below). We denote
the class of stochastic processes (\ref{3.11}) of finite superpositions with
marginal law $D$ by 
\begin{equation}
\mathbb{FS}_{m}\{D;\mathrm{E}X_{j}(t);\mathrm{Var}X_{j}(t)\}.  \label{FS}
\end{equation}

Define the mother process as the geometric process 
\begin{equation*}
\Lambda (t)=e^{X_{m\sup }(t)-c_{X}},c_{X}=\log \mathrm{E}e^{X_{m\sup
}(0)},M(\zeta )=\mathrm{E}e^{\zeta (X_{m\sup }(t)-c_{X})}, M_0(\zeta)=%
\mathrm{E}e^{\zeta X_{m\sup }(t)},
\end{equation*}%
where $X_{m\sup }(t),t\in \mathbb{R}_{+}$, is the finite superposition
process (\ref{3.12}). Note that 
\begin{equation*}
\log \mathrm{E}\exp \{\zeta _{1}X_{m\sup }(t_{1})+\zeta _{2}X_{m\sup
}(t_{2})\}=\sum_{j=1}^{m}\log \mathrm{E}\exp \{\zeta _{1}X_{j}(t_{1})+\zeta
_{2}X_{j}(t_{2})\},
\end{equation*}%
where $\log \mathrm{E}\exp \{\zeta _{1}X_{j}(t_{1})+\zeta
_{2}X_{j}(t_{2})\},j=1,..,m$ are given by (\ref{eq39}).

Denote 
\begin{equation}
M(\zeta _{1},\zeta _{2};t_{1}-t_{2})=\exp \left\{ -c_{X}(\zeta _{1}+\zeta
_{2})\right\} \mathrm{E}\exp \{\zeta _{1}X_{m\sup }(t_{1})+\zeta
_{2}X_{m\sup }(t_{2})\}  \label{3.13a}
\end{equation}

We can formulate the following theorem which can be proved similar to
Theorem~\ref{thm:ou}.

\begin{theorem}
\label{thm:ou1} Let $X_{m\sup }(t),t\in \mathbb{R}_{+}$ be a finite
superposition of OU-type stationary processes (\ref{3.12}) such that the L%
\'{e}vy measure $\nu $ in (\ref{3.1}) of the random variable $X_{m}(t)$
satisfies the condition that for a positive integer $q^{\ast }\in \mathbf{N}%
, $ 
\begin{equation}
\int_{\left\vert x\right\vert \geq 1}xe^{q^{\ast }x}\nu (dx)<\infty .
\label{F1}
\end{equation}%
Then, for any fixed $b$ such that 
\begin{equation}
b>\left\{ \frac{M_{0}(q^{\ast })}{M_{0}(1)^{q^{\ast }}}\right\} ^{\frac{1}{%
q^{\ast }-1}},  \label{F2}
\end{equation}%
stochastic processes 
\begin{equation*}
A_{n}(t)=\int_{0}^{t}\prod_{j=0}^{n}\Lambda ^{(j)}\left( sb^{j}\right) ds,
\end{equation*}%
converge in $\mathcal{L}_{q}$ to the stochastic process $A(t)\in \mathcal{L}%
_{q},$ as $n\rightarrow \infty $. The limiting process $A(t)$ satisfies 
\begin{equation*}
\mathrm{E}A^{q}(t)\sim t^{q-\log _{b}\mathrm{E}\Lambda ^{q}\left( t\right)
},\quad q\in \lbrack 0,q^{\ast }],t\in \lbrack 0,1].
\end{equation*}%
The scaling function is given by 
\begin{equation}
\varsigma (q)=q-\log _{b}\mathrm{E}\Lambda ^{q}\left( t\right) =q(1+\frac{%
c_{X}}{\log b})-\log _{b}M_{0}(q),\quad q\in \lbrack 0,q^{\ast }].
\label{F3}
\end{equation}%
In addition, 
\begin{equation}
\mathrm{Var}A(t)\geqslant 2t\int_{0}^{t}\left( 1-\frac{u}{t}\right) M(\zeta
_{1},\zeta _{2};u)du,  \label{F4}
\end{equation}

where the bivariate moment generating function $M(\zeta _{1},\zeta
_{2};t_{1}-t_{2})$ is given by (\ref{3.13a})
\end{theorem}

We are interested in generalization of the above result to the case of
infinite superposition of OU-type processes which has a long-range
dependence property.

Note that an infinite superposition $(m\rightarrow \infty )$ gives a
complete monotone class of covariance functions%
\begin{equation*}
R_{\sup }(t)=\int_{0}^{\infty }e^{-tu}dU(u),t\geq 0,
\end{equation*}%
for some finite measure $U,$ which display long-range dependence (see
Barndorff-Nielsen 1998, 2001, Barndorff-Nielsen and Leonenko 2005 for
possible covariance structures and spectral densities and Barndorff-Nielsen
et al. 2011 for multivariate generalizations).

We are going to consider an infinite superposition of the OU processes,
which corresponds to $m\rightarrow \infty $, that is now 
\begin{equation}
X_{\sup }(t)=\sum_{j=1}^{\infty }X_{j}(t)\text{,}  \label{sup1}
\end{equation}%
assuming that 
\begin{equation}
\sum_{j=1}^{\infty }\mathrm{E}X_{j}(t)<\infty ,\sum_{j=1}^{\infty }\mathrm{%
Var}X_{j}(t)<\infty ,  \label{sup2}
\end{equation}%
In this case 
\begin{equation}
R_{\sup }(t)=\sum_{j=1}^{\infty }\sigma _{j}^{2}\exp \left\{ -\lambda
_{j}\left\vert t\right\vert \right\} \text{,}  \label{sup3}
\end{equation}%
and if we assume that for some $\delta _{j}>0$ 
\begin{equation}
\mathrm{E}X_{j}(t)=\delta _{j}C_{1},\mathrm{Var}X_{j}(t)=\sigma
_{j}^{2}=\delta _{j}C_{2},\delta _{j}=j^{-(1+2(1-H))},\frac{1}{2}<H<1,
\label{sup01}
\end{equation}%
where the constants $C_{1}\in \mathbb{R}$ and $C_{2}>0$ represent some other
possible parameters (see examples in the Sections 8--13 below), then 
\begin{equation}
\mathrm{E}X_{\sup }(t)=C_{1}\sum_{j=1}^{\infty }\delta _{j}=C_{1}\zeta
(1+2(1-H))<\infty ,  \label{sup02}
\end{equation}%
where $\zeta (s)=\sum\limits_{n=1}^{\infty }\frac{1}{n^{s}},$ $\mathrm{Re}%
s>1,$ is the Riemann zeta-function, and with $\lambda _{j}=\lambda /j,$ we
have 
\begin{equation}
R_{\sup }(t)=\sum_{j=1}^{\infty }\sigma _{j}^{2}\exp \left\{ -\lambda
_{j}\left\vert t\right\vert \right\} =C_{2}\sum_{j=1}^{\infty }\delta
_{j}\exp \left\{ -\lambda \left\vert t\right\vert /j\right\} =\frac{%
L_{2}(\left\vert t\right\vert )}{\left\vert t\right\vert ^{2(1-H)}},\frac{1}{%
2}<H<1,  \label{sup03}
\end{equation}%
where $L_{2}$ is a slowly varying at infinity function, bounded on every
bounded interval. Thus we obtain a long range dependence property: 
\begin{equation*}
\int_{\mathbb{R}}R_{\sup }(t)dt=\infty .
\end{equation*}

We denote the class of stochastic processes (\ref{sup1}) of infinite
superpositions with marginal law $D$ as 
\begin{equation}
\mathbb{IS}\{D;\mathrm{E}X_{j}(t);\mathrm{Var}X_{j}(t)\}.  \label{IS}
\end{equation}

We are going to make an additional assumption that there exists parameters $%
\delta _{j}$ such that 
\begin{equation}
\mathrm{E}e^{\zeta X_{j}(0)}=\mathrm{E}e^{\zeta \bar{\delta}_{j}Y}
\label{eq:exp.superposition}
\end{equation}%
for some random variable $Y$. The sum 
\begin{equation}
\sum_{j=1}^{\infty }\bar{\delta}_{j}<\infty  \label{sup4}
\end{equation}%
must be finite. When we specialize (\ref{sup3}) to this situation we obtain 
\begin{equation}
R_{\sup }(t)=C_{2}\sum_{j=1}^{\infty }\bar{\delta}_{j}\exp \left\{ -\lambda
_{j}\left\vert t\right\vert \right\} ,  \label{eq:corr.inf.superposition}
\end{equation}%
for some $C_{2}>0.$ This approach allows also to treat the case of several
parameters.

We are starting our considerations with $\mathcal{L}_{q}$ convergence. Let 
\begin{equation*}
\rho _{j}(u_{1},\ldots ,u_{q-1})=\mathrm{E}%
e^{X_{j}(0)+X_{j}(u_{1})+X_{j}(u_{1}+u_{2})+\cdots +X_{j}(u_{1}+\cdots
+u_{q-1})}
\end{equation*}%
correspond to the process $X_{j}(t)$. Then, since $X_{j}(\cdot )$ are
independent of each other, 
\begin{equation}
\rho (u_{1},\ldots ,u_{q-1})=\prod_{j=1}^{\infty }\rho _{j}(u_{1},\ldots
,u_{q-1}).  \label{eq:repr.rho.infinite.super}
\end{equation}%
We have shown above that $\rho _{j}(u_{1},\ldots ,u_{q-1})$ is monotone
decreasing in $u_{1},\ldots ,u_{q-1}$. Therefore $\rho (u_{1},\ldots
,u_{q-1})$, being a product of monotone decreasing functions is monotone
decreasing itself. Next we prove finiteness of the series. For that recall
estimate (\ref{eq310}) 
\begin{equation*}
1\leq \rho _{j}(b^{n},\ldots ,b^{n})\leq \left( \frac{M_{0}(1+e^{-\lambda
_{j}b^{n}})}{M_{0}(1)M_{0}(e^{-\lambda _{j}b^{n}})}\right) ^{q}\leq \left( 
\frac{\mathrm{E}e^{2X_{j}(0)}}{\mathrm{E}e^{X_{j}(0)}e^{\mathrm{E}X_{j}(0)}}%
\right) ^{qe^{-\lambda _{j}b^{n}}}=C^{q\delta _{j}e^{-\lambda _{j}b^{n}}}
\end{equation*}%
where $C=\frac{\mathrm{E}e^{2Y}}{\mathrm{E}e^{Y}e^{\mathrm{E}Y}},$ and we
denote 
\begin{equation*}
M_{0j}(\zeta )=\mathrm{E}e^{\zeta X_{j}(t)}.
\end{equation*}

Then, using (\ref{eq:corr.inf.superposition}), we obtain 
\begin{equation*}
1\leq \rho (b^{n},\ldots ,b^{n})\leq C^{q\sum_{j=1}^{\infty }\delta
_{j}e^{-\lambda _{j}b^{n}}}=C^{\frac{q}{\sigma ^{2}}R_{sup}(b^{n})}\leq
1+o(1)\frac{q}{\sigma ^{2}}\ln CR_{sup}(b^{n}).
\end{equation*}%
Then $\sum_{n=1}^{\infty }\rho (b^{n},\ldots ,b^{n})$ is finite if the sum $%
\sum_{n=1}^{\infty }R_{sup}(b^{n})$ is finite.

We are left to check the mixing condition (\ref{eq:mixing}). But this
condition follows from the fact that it holds for $\rho _{j}$,
representation (\ref{eq:repr.rho.infinite.super}) and monotonicity of $\rho
_{j}$. Indeed let $1\leq i_{1}<i_{2}<\ldots \leq i_{m}$ and put $\mathbf{%
\delta }_{i_{1},\ldots ,i_{m}}(A)=(u_{1},\ldots ,u_{q-1})$, where $u_{i}=A$
if $i\in \{i_{1},\ldots ,i_{m}\}$ and $0$ otherwise. Let $N$ be a number
which we let tend to $\infty $ later. Then, for fixed $N$, 
\begin{equation*}
\prod_{j=1}^{N}\rho _{j}(\mathbf{\delta }_{i_{1},\ldots
,i_{m}}(A))\rightarrow \prod_{j=1}^{N}\mathrm{E}e^{i_{1}X_{j}(0)}\mathrm{E}%
e^{(i_{2}-i_{1})X_{j}(0)}\ldots \mathrm{E}e^{(q-i_{m})X_{j}(0)}
\end{equation*}%
by the corresponding property of the geometric Ornestein-Uhlenbeck process,
see (\ref{eq:ou.mixing}). The product, 
\begin{equation*}
1\leq \prod_{j=N}^{\infty }\rho _{j}(\mathbf{\delta }_{i_{1},\ldots
,i_{m}}(A))\leq \prod_{j=N}^{\infty }\rho _{j}(\mathbf{\delta }%
_{i_{1},\ldots ,i_{m}}(1))\rightarrow 1,
\end{equation*}%
as $N\rightarrow \infty $, uniformly in $A>1$. Therefore, the mixing
property holds.

Now we turn to proving the scaling property. For that we use Theorem~\ref%
{thm:scaling}. Let $M_{0}(\zeta)=\mathrm{E}e^{\zeta Y}$ be the moment
generating function of $Y$ and $f(\zeta)=\ln M_{0}(\zeta).$ Then, similarly
to \eqref{eq:ou.rhoq1}, 
\begin{multline*}
|\rho _{q}(s)|=1-\prod_{j=1}^{\infty }\left( \frac{M_{0}(q-1+e^{-\lambda
_{j}s})M_{0}(1)}{M_{0}(q)M_{0}(e^{-\lambda _{j}s})}\right) ^{\delta
_{j}}\leq 1-\prod_{j=1}^{\infty }e^{(-1+e^{-\lambda _{j}s})(f^{\prime
}(q)-f(1))\delta _{j}} \\
\leq 1-\prod_{j=1}^{\infty }e^{-\lambda _{j}s(f^{\prime }(q)-f(1))\delta
_{j}}\leq s(f^{\prime }(q)-f(1))\sum_{j=1}^{\infty }\lambda _{j}\delta _{j}.
\end{multline*}%
The convergence of the series immediately follows from this estimate and we
can apply \ Theorem~\ref{thm:scaling}.

Note that 
\begin{equation*}
\log \mathrm{E}\exp \{\zeta _{1}X_{\sup }(t_{1})+\zeta _{2}X_{\sup
}(t_{2})\}=\sum_{j=1}^{\infty }\log \mathrm{E}\exp \{\zeta
_{1}X_{j}(t_{1})+\zeta _{2}X_{j}(t_{2})\},
\end{equation*}%
where $\log \mathrm{E}\exp \{\zeta _{1}X_{j}(t_{1})+\zeta
_{2}X_{j}(t_{2})\},j=1,2,..$ are given by (\ref{eq39}).

Define the mother process as the geometric process 
\begin{equation*}
\Lambda (t)=e^{X_{\sup }(t)-c_{X}},c_{X}=\log \mathrm{E}e^{X_{\sup
}(0)},M(\zeta)=\mathrm{E}e^{\zeta(X_{\sup }(0)-c_{X})},
\end{equation*}%
where $X_{\sup }(t),t\in \mathbb{R}$, is the infinite superposition process (%
\ref{sup1}).

Denote 
\begin{equation}
M(\zeta _{1},\zeta _{2};t_{1}-t_{2})=\exp \left\{ -c_{X}(\zeta _{1}+\zeta
_{2})\right\} \mathrm{E}\exp \{\zeta _{1}X_{\sup }(t_{1})+\zeta _{2}X_{\sup
}(t_{2})\}.  \label{sup100}
\end{equation}%
We arrive to the following result.

\begin{theorem}
\label{thm:ou2} Let $X_{\sup }(t),t\in \mathbb{R}_{+}$ be an infinite
superposition of OU-type stationary processes (\ref{sup1}) such that (\ref%
{sup2}),(\ref{sup01}),(\ref{eq:exp.superposition}) are satisfied as well as (%
\ref{sup4}). Assume that the L\'{e}vy measure $\nu $ in (\ref{3.1}) of the
random variable $X_{\sup }(t)$ satisfies the condition that for a positive
integer $q^{\ast }\in \mathbf{N},$ 
\begin{equation}
\int_{\left\vert x\right\vert \geq 1}xe^{q^{\ast }x}\nu (dx)<\infty .
\label{sup7}
\end{equation}%
Then, for any fixed $b$ such that 
\begin{equation}
b>\left\{ \frac{M(q^{\ast })}{M(1)^{q^{\ast }}}\right\} ^{\frac{1}{q^{\ast
}-1}},  \label{sup8}
\end{equation}%
stochastic processes 
\begin{equation*}
A_{n}(t)=\int_{0}^{t}\prod_{j=0}^{n}\Lambda ^{(j)}\left( sb^{j}\right)
ds,t\in \lbrack 0,1]
\end{equation*}%
converge in $\mathcal{L}_{q}$ to the stochastic process $A(t)\in \mathcal{L}%
_{q},t\in \lbrack 0,1],$ as $n\rightarrow \infty $. The limiting process $%
A(t)$ satisfies 
\begin{equation*}
\mathrm{E}A^{q}(t)\sim t^{\zeta (q)},\quad q\in \lbrack 0,q^{\ast }].
\end{equation*}%
The scaling function is given by 
\begin{equation}
\varsigma (q)=q-\log _{b}\mathrm{E}\Lambda ^{q}\left( t\right) ,\quad q\in
\lbrack 0,q^{\ast }],t\in \lbrack 0,1].  \label{sup9}
\end{equation}%
In addition, 
\begin{equation}
\mathrm{Var}A(t)\geqslant \int_{0}^{t}\int_{0}^{t}M(\zeta _{1},\zeta
_{2};t_{1}-t_{2})dtds,  \label{sup10}
\end{equation}%
where the bivariate moment generating functions $M(\zeta _{1},\zeta
_{2};t_{1}-t_{2})$ is given by (\ref{sup100}).
\end{theorem}

\section{Log-tempered stable scenario}

This section introduces a scenario which generalize the log-inverse Gaussian
scenario obtained in Anh, Leonenko and Shieh (2008a, 2010). Note that the
tempered stable distribution (up to constants) arises in the theory of
Vershik-Yor subordinator (see Donati-Martin and Yor 2006, and the references
therein). This section constructs a multifractal process based on the
geometric tempered stable OU process. In this case, the mother process takes
the form $\Lambda (t)=\exp \left\{ X\left( t\right) -c_{X}\right\} ,$ where $%
X\left( t\right) ,t\geq 0$ is a stationary OU type process (\ref%
{eq:defn.ou.int}) with tempered stable marginal distribution and $c_{X}$ is
a constant depending on the parameters of its marginal distribution. This
form is needed for the condition $\mathrm{E}\Lambda \left( t\right) =1$ to
hold. The log-tempered stable scenario appeared in Novikov (1994) in a
physical setting and in Anh \emph{et al.} (2001) in a gemonic setting under
different terminology. So, we present here a rigorous proofs regarded these
scenarios. Some applications of the log-tempered stable scenario and other
related multifractal scenarios considered below in a subordinated models for
currency exchange rates can be found in Leonenko \emph{et al.} (2013).

We consider the stationary OU process whose marginal distribution is the
tempered stable distribution $TS(\kappa ,\delta ,\gamma )$ (see, for
example, Barndorff-Nielsen and Shephard 2002, Terdik and Woyczynski, 2004).
This distribution is the exponentially tilted version of the positive $%
\kappa $-stable law $S(\kappa ,\delta )$ whose cumulant transform is of the
form: $-\delta (2\zeta )^{\kappa },\zeta >0,\kappa \in (0,1),\delta >0.$We
denote its probability density function (pdf) as $s_{\kappa ,\delta
}(x),x>0. $ The pdf of the tempered stable distribution $TS(\kappa ,\delta
,\gamma )$ is%
\begin{equation}
\pi \left( x\right) =\pi (x;\kappa ,\delta ,\gamma )=e^{\delta \gamma
}s_{\kappa ,\delta }(x)e^{-\frac{x}{2}\gamma ^{1/\kappa }},x>0,\kappa \in
(0,1),\delta >0,\gamma >0.  \label{4.1a}
\end{equation}%
It is clear that $TS(\frac{1}{2},\delta ,\gamma )=IG\left( \delta ,\gamma
\right) ,$ the inverse Gaussian distribution with pdf%
\begin{equation*}
\pi \left( x\right) =\frac{1}{\sqrt{2\pi }}\frac{\delta e^{\delta \gamma }}{%
x^{3/2}}\exp \left\{ -\left( \frac{\delta ^{2}}{x}+\gamma ^{2}x\right) \frac{%
1}{2}\right\} \mathbf{1}_{[0,\infty )}(x),\delta >0,\gamma \geq 0.
\end{equation*}%
In general the pdf of ~tempered stable distribution is given in form of a
series representation (see, i.e., Anh et al. 2010).

The cumulant transform of a random variable $X\sim TS(\kappa ,\delta ,\gamma
)$ is of the form.%
\begin{equation}
\log \mathrm{E}e^{\zeta X}=\delta \gamma -\delta \left( \gamma ^{\frac{1}{%
\kappa }}-2\zeta \right) ^{\kappa },0<\zeta <\frac{\gamma ^{1/\kappa }}{2}.
\end{equation}%
Note that 
\begin{equation*}
\mathrm{E}X(t)=2\kappa \delta \gamma ^{\frac{\kappa -1}{\kappa }},\mathrm{Var%
}X(t)=4\kappa \left( 1-\kappa \right) \delta \gamma ^{\frac{\kappa -2}{%
\kappa }}.
\end{equation*}%
We will consider a stationary OU type process (\ref{eq:defn.ou}) with
marginal distribution $TS(\kappa ,\delta ,\gamma )$. This distribution is
self-decomposable (and hence infinitely divisible) with the L\'{e}vy triplet 
$(a,0,\nu ),$ where%
\begin{equation*}
\nu (du)=b(u)du,b(u)=2^{\kappa }\delta \frac{\kappa }{\Gamma \left( 1-\kappa
\right) }u^{-1-\kappa }e^{-\frac{u\gamma ^{1/\kappa }}{2}},u>0.
\end{equation*}%
The BDLP $Z(t)$ in (\ref{eq:defn.ou}) has a L\'{e}vy triplet ($\tilde{a},0,%
\tilde{\nu}),$\ with%
\begin{align}
\tilde{\nu}(du)& =\lambda \omega (u)du,  \label{4.40} \\
\omega (u)& =2^{\kappa }\delta \frac{\kappa }{\Gamma \left( 1-\kappa \right) 
}\left( \frac{\kappa }{u}+\frac{\gamma ^{1/\kappa }}{2}\right) u^{-\kappa
}e^{-\frac{u\gamma ^{1/\kappa }}{2}},u>0.  \label{4.41}
\end{align}

Consider a mother process of the form 
\begin{equation*}
\Lambda (t)=\exp \left\{ X\left( t\right) -c_{X}\right\}
\end{equation*}%
with%
\begin{equation}
c_{X}=\left[ \delta \gamma -\delta \left( \gamma ^{\frac{1}{\kappa }%
}-2\right) ^{\kappa }\right] ,\gamma >2^{\kappa },  \label{KO}
\end{equation}%
where $X\left( t\right) $ is OU processes with $TS(\kappa ,\delta ,\gamma )$
marginal distribution and correlation function $R_{X}\left( t\right) =\exp
\left\{ -\lambda \left\vert t\right\vert \right\} .$

The correlation function of the mother process takes the form 
\begin{equation}
\rho (\tau )=\frac{M(1,1;\tau )-1}{M(2)-1},\gamma >4^{\kappa },  \label{4.8}
\end{equation}%
where 
\begin{equation*}
M(\zeta )=e^{-\zeta c_{X}}\mathrm{E}e^{\zeta X(t)},
\end{equation*}%
and the bivariate moment generating function $M(1,1,\tau )$ is given by (\ref%
{OUM}), in which the L\'{e}vy measure $\tilde{\nu}$ is defined by (\ref{4.40}%
), (\ref{4.41}), and $c_{X}$ is given by (\ref{KO}).

Condition (\ref{eq:existence.derivative}) becomes%
\begin{equation*}
\int_{1}^{\infty }ue^{q^{\ast }u}u^{-1-\kappa }e^{-\frac{u\gamma ^{1/\kappa }%
}{2}}du=\int_{1}^{\infty }u^{-\kappa }e^{-u(\frac{\gamma ^{1/\kappa }}{2}%
-q^{\ast })}du<\infty ,
\end{equation*}%
if $0<q^*<\frac{\gamma ^{1/\kappa }}{2},\kappa \in (0,1).$ Note that $M(q)$
exists if $(2q)^{\kappa }<\gamma .$

We can formulate the following

\begin{theorem}
\label{thm:TS}Let $X\left( t\right) $ be an OU processes with $TS(\kappa
,\delta ,\gamma )$ marginal distributions, $\lambda >0$ and 
\begin{equation*}
Q=\{q:0<q<\frac{\gamma ^{1/\kappa }}{2},\gamma \geq \max \{(2q^{\ast
})^{\kappa },4^{\kappa }\},\kappa \in (0,1),\delta >0\}\cap \lbrack
0,q^{\ast }],
\end{equation*}%
where $q^{\ast }$ is a fixed integer. Then, for any 
\begin{equation*}
b>\exp \left\{ -\gamma \delta +\frac{\delta }{1-q^{\ast }}\left( \gamma ^{%
\frac{1}{\kappa }}-2q^{\ast }\right) ^{\kappa }-\frac{q^{\ast }}{1-q^{\ast }}%
\delta \left( \gamma ^{\frac{1}{\kappa }}-2\right) ^{\kappa }\right\}
,\gamma \geq \max \{(2q^{\ast })^{\kappa },4^{\kappa }\},
\end{equation*}%
the stochastic processes 
\begin{equation*}
A_{n}(t)=\int_{0}^{t}\prod_{j=0}^{n}\Lambda ^{(j)}(sb^{j})ds,\;t\in \lbrack
0,1]
\end{equation*}%
converge in $\mathcal{L}_{q}$ to the stochastic process $A(t)$ for each fixed%
$\;t\in \lbrack 0,1]$ as $n\rightarrow \infty $ such that, $A_{q}\left(
1\right) \in \mathcal{L}_{q}$, for $q\in Q,$ and 
\begin{equation*}
\mathrm{E}A^{q}(t)\sim t^{\varsigma (q)},\;\;
\end{equation*}%
where the scaling function $\varsigma \left( q\right) $ is given by 
\begin{equation*}
\varsigma (q)=q\left( 1+\frac{\delta \gamma }{\log b}-\frac{\delta }{\log b}%
\left( \gamma ^{\frac{1}{\kappa }}-2\right) \right) ^{\kappa }+\frac{\delta 
}{\log b}\left( \gamma ^{\frac{1}{\kappa }}-2q\right) ^{\kappa }-\frac{%
\delta \gamma }{\log b},q\in Q.
\end{equation*}%
Moreover,%
\begin{equation*}
\text{\textrm{Var}}A\left( t\right) \geqslant \int_{0}^{t}\int_{0}^{t}\left[
M(1,1;u-w)-1\right] dudw,
\end{equation*}%
where $M(1,1,\tau )$ is given by (\ref{eq39}), in which the L\'{e}vy measure 
$\tilde{\nu}$ is defined by (\ref{4.40}), (\ref{4.41}).
\end{theorem}

Theorem~\ref{thm:TS} follows from the Theorem~\ref{thm:ou}. Note that for $%
\kappa =1/2$ Theorem~\ref{thm:TS}\ is an extension of the Theorem 4 of Anh,
Leonenko and Shieh (2008).

In this particular case we arrive to log-inverse Gaussian \ scenario where
the scaling function is of the form: 
\begin{equation*}
\varsigma \left( q\right) =q\left( 1+\frac{\delta \left[ \gamma -\sqrt{%
\gamma ^{2}-2}\right] }{\log b}\right) +\frac{\delta }{\log b}\sqrt{\gamma
^{2}-2q}-\frac{\gamma \delta }{\log b},q\in Q,
\end{equation*}%
and 
\begin{equation}
Q=\{q:0<q<\frac{\gamma ^{2}}{2},\gamma \geq 2,\delta >0\}\cap (0,q^{\ast })
\label{Q1}
\end{equation}%
if 
\begin{equation*}
b>\exp \left\{ -\gamma \delta -\frac{\delta }{1-q^{\ast }}\sqrt{\gamma
^{2}-2q}-\frac{q^{\ast }}{1-q^{\ast }}\delta \sqrt{\gamma ^{2}-2}\right\}
\end{equation*}%
and $q^{\ast }$ is a fixed integer.

Note that the set (\ref{Q1}) is an extension of the log-inverse Gaussian
scenario in the Theorem 4 of Anh, Leonenko and Shieh (2008a) which is
obtained for the set

\begin{equation*}
Q=\{q:0<q<\frac{\gamma ^{2}}{2},\gamma \geq 2,\delta >0\}\cap \lbrack 1,2].
\end{equation*}%
In this case the R\'{e}nyi function $T(q)=\varsigma (q)+1.$

We can construct log-tempered stable scenarios for a more general class of
finite superpositions of stationary tempered stable OU-type processes (\ref%
{3.12}), where $X_{j}(t),j=1,...,m,$ are independent stationary processes
with marginals $X_{j}(t)\thicksim TS(\kappa ,\delta _{j},\gamma ),j=1,...,m,$
and parameters $\delta _{j},j=1,...,m.$ Using notation (\ref{FS}), we
consider the class of processes 
\begin{equation*}
\mathbb{FS}_{m}\{TS(\kappa ,\sum_{j=1}^{m}\delta _{j},\gamma );2\kappa
\delta _{j}\gamma ^{\frac{\kappa -1}{\kappa }};4\kappa \left( 1-\kappa
\right) \gamma ^{\frac{\kappa -2}{\kappa }}\delta _{j})
\end{equation*}

It follows from the Theorem~\ref{thm:ou1} that the statement of Theorem~\ref%
{thm:TS} can be reformulated for $X_{m\sup }(t)$ with $\delta
=\sum_{j=1}^{m}\delta _{j}$ ,$~$and 
\begin{equation*}
M(\zeta _{1},\zeta _{2};t_{1}-t_{2})=\exp \left\{ -c_{X}(\zeta _{1}+\zeta
_{2})\right\} \mathrm{E}\exp \{\zeta _{1}X_{m\sup }(t_{1})+\zeta
_{2}X_{m\sup }(t_{2})\},
\end{equation*}%
where 
\begin{equation*}
\log \mathrm{E}\exp \{\zeta _{1}X_{m\sup }(t_{1})+\zeta _{2}X_{m\sup
}(t_{2})\}=\sum_{j=1}^{m}\log \mathrm{E}\exp \{\zeta _{1}X_{j}(t_{1})+\zeta
_{2}X_{j}(t_{2})\},
\end{equation*}%
and $\log \mathrm{E}\exp \{\zeta _{1}X_{j}(t_{1})+\zeta
_{2}X_{j}(t_{2})\},j=1,..,m$ are given by (\ref{eq39}).

Moreover, one can construct log-tempered stable scenarios for a more general
class of infinite superpositions of stationary tempered stable OU-type
processes (\ref{sup1}), where $X_{j}(t),j=1,...,m,$ are independent
stationary processes with marginals $X_{j}(t)\thicksim TS(\kappa ,\delta
_{j},\gamma ),j=1,2..$ and parameters $\delta _{j},j=1,2....$ .Using
notation (\ref{IS}), we consider the class of processes 
\begin{equation*}
\mathbb{IS\{}TS(\kappa ,\sum_{j=1}^{\infty }\delta _{j},\gamma );2\kappa
\delta _{j}\gamma ^{\frac{\kappa -1}{\kappa }};4\kappa \left( 1-\kappa
\right) \gamma ^{\frac{\kappa -2}{\kappa }}\delta _{j}\}.
\end{equation*}

It follows from the Theorem~\ref{thm:ou2} that the statement of the Theorem~%
\ref{thm:TS} remains true with $\delta =\sum_{j=1}^{\infty }\delta _{j},$
and 
\begin{equation*}
M(\zeta _{1},\zeta _{2};t_{1}-t_{2})=\exp \left\{ -c_{X}(\zeta _{1}+\zeta
_{2})\right\} \mathrm{E}\exp \{\zeta _{1}X_{\sup }(t_{1})+\zeta _{2}X_{\sup
}(t_{2})\},
\end{equation*}%
where 
\begin{equation*}
\log \mathrm{E}\exp \{\zeta _{1}X_{\sup }(t_{1})+\zeta _{2}X_{\sup
}(t_{2})\}=\sum_{j=1}^{\infty }\log \mathrm{E}\exp \{\zeta
_{1}X_{j}(t_{1})+\zeta _{2}X_{j}(t_{2})\},
\end{equation*}%
and $\log \mathrm{E}\exp \{\zeta _{1}X_{j}(t_{1})+\zeta
_{2}X_{j}(t_{2})\},j=1,..,m$ are given by (\ref{eq39}) with L\'{e}vy triplet 
$(\widetilde{a},0,\tilde{\nu})$ given by (\ref{4.40}) and (\ref{4.41}).

\section{Log-normal tempered stable scenario}

This subsection constructs a multifractal process based on the geometric
normal tempered stable (NTS) OU process. The log-normal tempered stable
scenario is important for risky asset modelling, see Leonenko, Petherick and
Taufer (2013).

We consider a random variable $X=\mu +\beta Y+\sqrt{Y}\epsilon ,$ where the
random variable $Y$ follows the $TS(\kappa ,\delta ,\gamma )$ distribution, $%
\epsilon $ has a standard normal distribution, and $Y$ and $\epsilon $ are
independent. We then say that $X$ follows the normal tempered stable law $%
NTS(\kappa ,\gamma ,\beta ,\mu ,\delta )$ (see, for example,
Barndorff-Nielsen and Shephard 2002). In particular, for $\kappa =1/2$ we
have that $NTS(\frac{1}{2},\gamma ,\beta ,\mu ,\delta )$ is the same as the
normal inverse Gaussian law $NIG(\alpha ,\beta ,\mu ,\delta )$ with $\alpha =%
\sqrt{\beta ^{2}+\gamma ^{2}}$ (see Barndorff-Nielsen 1998). We assume that%
\begin{equation*}
\mu \in \mathbb{R}\text{, }\mathbb{\delta >}0\text{, }\gamma >0,\beta
>0,\kappa \in (0,1).
\end{equation*}%
It was pointed out by Barndorff-Nielsen and Shephard (2002) that $NTS(\kappa
,\gamma ,\beta ,\mu ,\delta )$ \ is self-decomposable. Thus, there exists a
stationary OU-type process $X(t),t\geq 0,$ with stationary $NTS(\kappa
,\gamma ,\beta ,\mu ,\delta )$ marginal distribution and the correlation
function $r_{X}\left( t\right) =\exp \left\{ -\lambda \left\vert
t\right\vert \right\} .$Note that 
\begin{equation*}
\mathrm{E}X(t)=\mu +2\kappa \beta \delta \gamma ^{\frac{\kappa -1}{\kappa }},%
\mathrm{Var}X(t)=2\kappa \delta \gamma ^{\frac{\kappa -1}{\kappa }}-4\kappa
\beta ^{2}\delta \left( \kappa -1\right) \gamma ^{\frac{\kappa -2}{\kappa }}.
\end{equation*}%
We see that the variance can be factorized in a similar manner as in Section
7, and thus superposition can be used to create multifractal scenarios with
more elaborate dependence structures.

The cumulant transform of the random variable $X(t)$ with $NTS(\kappa
,\gamma ,\beta ,\mu ,\delta )$ distribution is equal to%
\begin{equation}
\log \mathrm{E}e^{\zeta X(t)}=\mu \zeta +\delta \gamma -\delta \left( \alpha
^{2}-(\beta +\zeta )^{2}\right) ^{\kappa },\left\vert \beta +\zeta
\right\vert <\alpha =\sqrt{\beta ^{2}+\gamma ^{1/\kappa }}.  \label{4.10}
\end{equation}%
The L\'{e}vy triplet of $NTS(\kappa ,\gamma ,\beta ,\mu ,\delta )$ is $%
(a,0,\nu ),$ where%
\begin{equation}
\nu (du)=b(u)du,  \label{4.11}
\end{equation}%
\begin{equation}
b(u)=\frac{\delta }{\sqrt{2\pi }}\alpha ^{\kappa +\frac{1}{2}}\frac{\kappa
2^{\kappa +1}}{\Gamma \left( 1-\kappa \right) }\left\vert u\right\vert
^{-(\kappa +\frac{1}{2})}K_{\kappa +\frac{1}{2}}(\alpha \left\vert
u\right\vert )e^{\beta u},u\in \mathbb{R},  \label{Levy1}
\end{equation}%
where here and below $_{\nu }(z)=\int_{0}^{\infty }e^{-z\cosh (u)}\cosh (\nu
z)du,\;z>0,$ is the modified Bessel function of the third kind of index $\nu
,\mathrm{Re}\nu >0.$

From (\ref{3.6}), (\ref{4.10}) and the formulae%
\begin{equation*}
K_{\nu }(x)=K_{\nu }(-x),~K_{-\nu }(x)=K_{\nu }(x),\frac{d}{dx}K_{\nu }(x)=-%
\frac{\lambda }{x}K_{\nu }(x)-K_{\nu -1}(x),
\end{equation*}%
we obtain that the BDLP $Z(t)$ in (\ref{3.4}) has a L\'{e}vy triplet ($%
\tilde{a},0,\tilde{\nu}),$\ with%
\begin{equation*}
\tilde{\nu}(du)=\lambda \omega (u)du,
\end{equation*}%
where 
\begin{align}
\omega (u)& =-b(u)-ub^{\prime }(u)=\frac{\delta }{\sqrt{2\pi }}\alpha
^{\kappa +\frac{1}{2}}\frac{\kappa 2^{\kappa +1}}{\Gamma \left( 1-\kappa
\right) }  \notag \\
& \hspace{1cm}\times \{(\kappa -\frac{1}{2})\left\vert u\right\vert
^{-(\kappa +\frac{1}{2})}K_{\kappa +\frac{1}{2}}(\alpha \left\vert
u\right\vert )e^{\beta u}+\left\vert u\right\vert ^{-(\kappa -\frac{1}{2})}[-%
\frac{\kappa +\frac{1}{2}}{\left\vert u\right\vert }K_{\kappa +\frac{1}{2}%
}(\alpha \left\vert u\right\vert )e^{\beta u}  \notag \\
& \hspace{1cm}-K_{\kappa -\frac{1}{2}}(\alpha \left\vert u\right\vert
)e^{\beta u}\alpha +K_{\kappa +\frac{1}{2}}(\alpha \left\vert u\right\vert
)e^{\beta u}\beta ]\}.  \label{4.12}
\end{align}

Consider a mother process of the form 
\begin{equation*}
\Lambda (t)=\exp \left\{ X\left( t\right) -c_{X}\right\} ,
\end{equation*}%
with%
\begin{equation*}
c_{X}=\mu +\delta \gamma -\delta \left( \beta ^{2}+\gamma ^{1/\kappa
}-(\beta +1)^{2}\right) ^{\kappa },\beta <\frac{\gamma ^{1/\kappa }-1}{2},
\end{equation*}%
where $X\left( t\right) $ is a stationary $NTS(\kappa ,\gamma ,\beta ,\mu
,\delta )$ OU-type process.

Under condition B$^{\prime \prime },$ we obtain the following moment
generating function%
\begin{equation}
M\left( \zeta \right) =\mathrm{E}\exp \left\{ \zeta \left( X\left( t\right)
-c_{X}\right) \right\} =e^{-c_{X}\zeta }e^{\mu \zeta +\delta \gamma -\delta
\left( \beta ^{2}+\gamma ^{1/\kappa }-(\beta +\zeta )^{2}\right) ^{\kappa
}\,},\text{\quad }\,\left\vert \beta +\zeta \right\vert <\alpha ,
\label{4.13}
\end{equation}%
and bivariate moment generating function 
\begin{align}
M\left( \zeta _{1},\zeta _{2};\left( t_{1}-t_{2}\right) \right) &=\mathrm{E}%
\exp \left\{ \zeta _{1}\left( X(t_{1}\right) -c_{X})+\zeta _{2}\left(
X(t_{2}\right) -c_{X})\right\}  \notag \\
&=e^{-c_{X}\left( \zeta _{1}+\zeta _{2}\right) }\mathrm{E}\exp \left\{ \zeta
_{1}X(t_{1})+\zeta _{2}\left( X(t_{2}\right) \right\} ,  \label{4.14}
\end{align}%
and $\mathrm{E}\exp \left\{ \zeta _{1}X(t_{1})+\zeta _{2}\left(
X(t_{2}\right) \right\} $ is given by (\ref{eq39}) with L\'{e}vy measure $%
\tilde{\nu}$ \ having density (\ref{4.12}). Thus, the correlation function
of the mother process takes the form 
\begin{equation}
\rho (\tau )=\frac{M(1,1;\tau )-1}{M(2)-1},  \label{4.15}
\end{equation}%
where we assumed that $\beta <(\gamma ^{1/\kappa }-4)/4.$

Note that as $z\rightarrow \infty $ the modified Bessel function of the
third kind of index $\nu .$ 
\begin{equation*}
K_{\nu }(z)=\sqrt{\frac{\pi }{2}}z^{-1/2}e^{-z}(1+\frac{4\nu ^{2}-1}{8z}%
+...),z>0,
\end{equation*}

Condition (\ref{eq:existence.derivative}) now becomes%
\begin{equation*}
\int_{\left\vert u\right\vert >1}ue^{q^{\ast }u}\left\vert u\right\vert
^{-(\kappa +\frac{1}{2})}K_{\kappa +\frac{1}{2}}(\alpha \left\vert
u\right\vert )e^{\beta u}du<\infty ,if\left\vert \beta +q^{\ast }\right\vert
<\alpha =\sqrt{\beta ^{2}+\gamma ^{1/\kappa }}.
\end{equation*}

\begin{theorem}
\label{thm:NTS}Let $X\left( t\right) $ be a stationary $NTS(\kappa ,\gamma
,\beta ,\mu ,\delta )$ OU-type process, $\lambda >0$ and 
\begin{equation*}
q\in Q=\left\{ q:0<q<q^{\ast }\leq \sqrt{\beta ^{2}+\gamma ^{1/\kappa }}%
-\beta ,\beta <(\gamma ^{1/\kappa }-1)/2,\mu \in \mathbb{R}\text{, }\mathbb{%
\delta >}0,\kappa \in (0,1)\right\} ,
\end{equation*}%
where $q^{\ast }$ is a fixed integer.

Then, for any 
\begin{equation*}
b>\exp \left\{ -\delta \gamma +\frac{\delta \left( \beta ^{2}+\gamma
^{1/\kappa }-(\beta +q^{\ast })^{2}\right) ^{\kappa }-q^{\ast }\delta \left(
\beta ^{2}+\gamma ^{1/\kappa }-(\beta +1)^{2}\right) ^{\kappa }}{1-q^{\ast }}%
\right\} ,
\end{equation*}%
the sequence of stochastic processes 
\begin{equation*}
A_{n}(t)=\int_{0}^{t}\prod_{j=0}^{n}\Lambda ^{(j)}(sb^{j})ds,\;t\in \lbrack
0,1]
\end{equation*}%
converge in $\mathcal{L}_{q}$ to the stochastic process $A(t)$ for each fixed%
$\;t\in \lbrack 0,1]$ as $n\rightarrow \infty $ such that $A\left( 1\right)
\in \mathcal{L}_{q}$ for $q\in Q,$ and 
\begin{equation*}
\mathrm{E}A^{q}(t)\sim t^{\varsigma (q)},\;\;q\in Q,
\end{equation*}%
where the scaling function $\varsigma \left( q\right) $ is given by 
\begin{align*}
\varsigma (q)& =\left( 1-\frac{\delta \left[ \left( \beta ^{2}+\gamma
^{1/\kappa }-(\beta +1)^{2}\right) \right] ^{\kappa }-\gamma }{\log b}%
\right) q+ \\
& \hspace{1cm}+\frac{\delta }{\log b}\left( \beta ^{2}+\gamma ^{1/\kappa
}-(\beta +q)^{2}\right) ^{\kappa }-\frac{\delta \gamma }{\log b},q\in Q.
\end{align*}%
Moreover,%
\begin{equation*}
\text{\textrm{Var}}A\left( t\right) \geqslant \int_{0}^{t}\int_{0}^{t}\left[
M(1,1;u-w)-1\right] dudw,
\end{equation*}%
where $M$ is given by (\ref{4.14}).
\end{theorem}

Theorem~\ref{thm:NTS} follows from the Theorem~\ref{thm:ou}.

Note that, for $\kappa =1/2,$ Theorem~\ref{thm:NTS} is an extension to
Theorem~5 of Anh, Leonenko and Shieh (2008a), which is now extended to
present a log-normal inverse Gaussian scenarios with scaling function$:$%
\begin{align*}
\varsigma (q)& =\left( 1-\frac{\delta \left[ \sqrt{\beta ^{2}+\gamma
^{2}-(\beta +1)^{2}}-\gamma \right] }{\log b}\right) q+ \\
& \hspace{1cm}+\frac{\delta }{\log b}\sqrt{\beta ^{2}+\gamma ^{2}-(\beta
+q)^{2}}-\frac{\delta \gamma }{\log b},q\in Q,
\end{align*}%
where we use the notation $NTS(\frac{1}{2},\gamma ,\beta ,\mu ,\delta
)=NIG(\alpha ,\beta ,\delta ,\mu ),$ $\gamma =\sqrt{\alpha ^{2}-\beta ^{2}},$
\begin{equation*}
Q=\left\{ q:0<q<q^{\ast }\leq \sqrt{\beta ^{2}+\gamma ^{2}}-\beta ,\beta
<(\gamma ^{2}-1)/2,\mu \in \mathbb{R}\text{, }\mathbb{\delta >}0\right\} ,
\end{equation*}%
and 
\begin{equation*}
b>\exp \left\{ -\delta \gamma +\frac{\delta \sqrt{\beta ^{2}+\gamma
^{2}-(\beta +q^{\ast })^{2}}-q^{\ast }\delta \sqrt{\beta ^{2}+\gamma
^{2}-(\beta +1)^{2}}}{1-q^{\ast }}\right\} .
\end{equation*}

\bigskip We can construct log-normal tempered stable scenarios for a more
general class of finite superpositions of stationary tempered stable OU-type
processes (\ref{3.12}), where $X_{j}(t),j=1,...,m,$ are independent
stationary processes with marginals $X_{j}(t)\thicksim NTS(\kappa ,\gamma
,\beta ,\mu _{j},\delta _{j}),$ $j=1,...,m,$ and parameters $\mu _{j},\delta
_{j},j=1,...,m.$. Using notation (\ref{FS}), we consider the class of
processes 
\begin{equation*}
\mathbb{FS}_{m}\{NTS(\kappa ,\gamma ,\beta ,\sum_{j=1}^{m}\mu
_{j},\sum_{j=1}^{m}\delta _{j});\mu _{j}+2\kappa \beta \delta _{j}\gamma ^{%
\frac{\kappa -1}{\kappa }};\left[ 2\kappa \gamma ^{\frac{\kappa -1}{\kappa }%
}-4\kappa \beta ^{2}\left( \kappa -1\right) \gamma ^{\frac{\kappa -2}{\kappa 
}}\right] \delta _{j}\}
\end{equation*}

Then the statement of the Theorem~\ref{thm:NTS} can be reformulated for $%
X_{m\sup }$ with $\mu =\sum_{j=1}^{m}\mu _{j},\delta =\sum_{j=1}^{m}\delta
_{j},$ and 
\begin{align*}
M\left( \zeta _{1},\zeta _{2};\left( t_{1}-t_{2}\right) \right) & =\mathrm{E}%
\exp \left\{ \zeta _{1}\left( X_{m\sup }(t_{1}\right) -c_{X})+\zeta
_{2}\left( X_{m\sup }(t_{2}\right) -c_{X})\right\} \\
& =e^{-c_{X}\left( \zeta _{1}+\zeta _{2}\right) }\mathrm{E}\exp \left\{
\zeta _{1}X_{m\sup }(t_{1})+\zeta _{2}\left( X_{m\sup }(t_{2}\right)
\right\} ,
\end{align*}%
where%
\begin{equation*}
\log \mathrm{E}\exp \{\zeta _{1}X_{m\sup }(t_{1})+\zeta _{2}X_{m\sup
}(t_{2})\}=\sum_{j=1}^{m}\log \mathrm{E}\exp \{\zeta _{1}X_{j}(t_{1})+\zeta
_{2}X_{j}(t_{2})\},
\end{equation*}%
and $\log \mathrm{E}\exp \{\zeta _{1}X_{j}(t_{1})+\zeta
_{2}X_{j}(t_{2})\},j=1,..,m$ are given by (\ref{eq39}) with L\'{e}vy measure 
$\tilde{\nu}$ \ having density (\ref{4.12}).

Moreover, one can construct log-normal tempered stable scenarios for a more
general class of infinite superpositions of stationary normal tempered
stable OU-type processes (\ref{sup1}), where $X_{j}(t),j=1,...,m,$ are
independent stationary processes with marginals $X_{j}(t)\thicksim
NTS(\kappa ,\gamma ,\beta ,\mu _{j},\delta _{j}),j=1,2..$ and parameters $%
\mu _{j},\delta _{j},j=1,2....$ . Using notation (\ref{IS}), we consider the
class of processes 
\begin{equation*}
\mathbb{IS}\{NTS(\kappa ,\gamma ,\beta ,\sum_{j=1}^{\infty }\mu
_{j},\sum_{j=1}^{\infty }\delta _{j});\mu _{j}+2\kappa \beta \delta
_{j}\gamma ^{\frac{\kappa -1}{\kappa }};\left[ 2\kappa \gamma ^{\frac{\kappa
-1}{\kappa }}-4\kappa \beta ^{2}\left( \kappa -1\right) \gamma ^{\frac{%
\kappa -2}{\kappa }}\right] \delta _{j}\}.
\end{equation*}%
Then the statement of Theorem~\ref{thm:NTS} remains true with $\delta
=\sum_{j=1}^{\infty }\delta _{j},\mu =\sum_{j=1}^{\infty }\mu _{j}$ and 
\begin{align*}
M\left( \zeta _{1},\zeta _{2};\left( t_{1}-t_{2}\right) \right) & =\mathrm{E}%
\exp \left\{ \zeta _{1}\left( X_{\sup }(t_{1}\right) -c_{X})+\zeta
_{2}\left( X_{\sup }(t_{2}\right) -c_{X})\right\} \\
& =e^{-c_{X}\left( \zeta _{1}+\zeta _{2}\right) }\mathrm{E}\exp \left\{
\zeta _{1}X_{\sup }(t_{1})+\zeta _{2}\left( X_{\sup }(t_{2}\right) \right\} ,
\end{align*}%
where 
\begin{equation*}
\log \mathrm{E}\exp \{\zeta _{1}X_{\sup }(t_{1})+\zeta _{2}X_{\sup
}(t_{2})\}=\sum_{j=1}^{\infty }\log \mathrm{E}\exp \{\zeta
_{1}X_{j}(t_{1})+\zeta _{2}X_{j}(t_{2})\},
\end{equation*}%
and $\log \mathrm{E}\exp \{\zeta _{1}X_{j}(t_{1})+\zeta
_{2}X_{j}(t_{2})\},j=1,..,m,...$ are given by (\ref{3.8}) with L\'{e}vy
measure $\tilde{\nu}$ \ having density (\ref{4.12}).

\section{Log-gamma scenario}

The log-gamma multifractal scenario is well-known in the theory of
turbulence and multiplicative cascades (Saito 1992). \ In this section, we
propose a stationary version of the log-gamma scenario. We will use a
stationary OU type process $X(t),t\in \mathbb{R}_{+},$(see, \ref{3.4})) with
marginal gamma distribution $\Gamma (\beta ,\alpha )$. It is known that the
gamma distribution with the moment generating function $\mathrm{E}\exp
\{\zeta X(t)\}=\left( 1-\frac{\zeta }{\alpha }\right) ^{-\beta },$~$\zeta
<\alpha ,\alpha >0,\beta >0,$ is self-decomposable. The L\'{e}vy triplet is
of the form $\left( 0,0,\nu \right) $, where 
\begin{equation*}
\nu (du)=\frac{\beta e^{-\alpha u}}{u}\mathbf{1}_{[0,\infty )}(u)du,
\end{equation*}%
while the BDLP $Z(t)$ in (\ref{3.4}) is a compound Poisson subordinator,
that is

\begin{equation*}
\kappa (z)=\log \mathrm{E}e^{izZ(1)}=\frac{i\beta z}{\alpha -iz},z\in 
\mathbb{R}\text{,}
\end{equation*}%
and the (finite) L\'{e}vy measure $\tilde{\nu}$ of $Z(1)$ is%
\begin{equation}
\tilde{\nu}(du)=\alpha \beta e^{-\alpha u}\mathbf{1}_{(0,\infty )}(u)du.
\label{eq:LevyKhinchine.Poisson.subordinator}
\end{equation}%
The covariance function is then $r_{X}\left( t\right) =(\beta /\alpha
^{2})\exp \left( -\lambda \left\vert t\right\vert \right) .$

Consider a mother process of the form 
\begin{equation*}
\Lambda (t)=\exp \left( X\left( t\right) -c_{X}\right) ,\text{\quad }%
c_{X}=\log \frac{1}{\left( 1-\frac{1}{\alpha }\right) ^{\beta }},\alpha >1,
\end{equation*}%
where $X\left( t\right) $ is a stationary gamma OU type stochastic process.

We obtain the following moment generating function: 
\begin{equation}
M\left( \zeta \right) =\mathrm{E}\exp \left( \zeta \left( X\left( t\right)
-c_{X}\right) \right) =\frac{e^{-c_{X}\zeta }}{\left( 1-\frac{\zeta }{\alpha 
}\right) ^{\beta }},\text{~}\zeta <\alpha ,\alpha >1,  \label{4.3}
\end{equation}%
and the bivariate moment generating function is given by the formula (\ref%
{eq39}), in which the measure $\tilde{\nu}$ is given by (\ref%
{eq:LevyKhinchine.Poisson.subordinator}), since 
\begin{align}
M\left( \zeta _{1},\zeta _{2};\left( t_{1}-t_{2}\right) \right) & =\mathrm{E}%
\exp \left( \zeta _{1}\left( X(t_{1}\right) -c_{X})+\zeta _{2}\left(
X(t_{2}\right) -c_{X})\right)  \notag \\
& =e^{-c_{X}\left( \zeta _{1}+\zeta _{2}\right) }\mathrm{E}\exp \left( \zeta
_{1}X(t_{1})+\zeta _{2}\left( X(t_{2}\right) \right)  \notag \\
& =e^{-c_{X}\left( \zeta _{1}+\zeta _{2}\right) }\exp \left( \int_{\mathbb{R}%
}\frac{\beta \sum_{j=1}^{2}\zeta _{j}e^{-\lambda (t_{j}-s)}\mathbf{1}%
_{[0,\infty )}(t_{j}-s)}{\alpha -\sum_{j=1}^{2}\zeta _{j}e^{-\lambda
(t_{j}-s)}\mathbf{1}_{[0,\infty )}(t_{j}-s)}ds\right) ,  \label{4.4}
\end{align}%
or 
\begin{align}
M\left( \zeta _{1},\zeta _{2};\left( t_{1}-t_{2}\right) \right) & =\exp
\left( -c_{X}\left( \zeta _{1}+\zeta _{2}\right) \right)  \notag \\
& \times \exp \biggl(\int_{\mathbb{R}}\int_{\mathbb{R}}\biggl(\exp \biggl(%
u\sum_{j=1}^{2}\zeta _{j}e^{-\lambda (t_{j}-s)}\mathbf{1}_{[0,\infty
)}(t_{j}-s)\biggr)-1  \notag \\
& -u\biggl(\sum_{j=1}^{2}\zeta _{j}e^{-\lambda (t_{j}-s)}\mathbf{1}%
_{[0,\infty )}(t_{j}-s)\biggl)\mathbf{1}_{\left[ -1,1\right] }\left(
u\right) \biggr)\alpha \beta e^{-\alpha u}\mathbf{1}_{(0,\infty )}(u)duds%
\biggr).  \label{4.5}
\end{align}%
Thus, the correlation function of the mother process takes the form 
\begin{equation*}
\rho (\tau )=\frac{M(1,1;\tau )-1}{M(2)-1},
\end{equation*}%
where $M(2)$ is given by (\ref{4.3}) and $M(1,1;\tau )$ is given by (\ref%
{4.5}). It turns out that, in this case, 
\begin{equation*}
\log _{b}\mathrm{E}\Lambda \left( t\right) ^{q}=\frac{1}{\log b}\left(
-q\log \frac{1}{\left( 1-\frac{1}{\alpha }\right) ^{\beta }}-\beta \log
\left( 1-\frac{q}{\alpha }\right) \right) ,
\end{equation*}%
and the condition (\ref{eq:existence.derivative}) of Theorem~\ref{thm:ou}
holds, since 
\begin{equation*}
\int_{\left\vert u\right\vert \geq 1}ue^{q^{\ast }u}\nu (du)=\frac{\alpha
^{\beta }\beta }{\Gamma \left( \beta \right) }\int_{1}^{\infty }e^{q^{\ast
}u}e^{-\alpha u}du<\infty ,q^{\ast }<\alpha .
\end{equation*}%
We formulate the following

\begin{theorem}
\label{thm:G}Let $X\left( t\right) $ be a stationary gamma OU type
stochastic process and let $Q=\{q:0<q<q^{\ast }\,<\alpha ,\alpha >2,\beta
>0\}$, where $q^{\ast }$ is a fixed integer. Then, for any%
\begin{equation*}
b>\left[ \left( 1-\frac{1}{\alpha }\right) ^{\beta q^{\ast }}/\left( 1-\frac{%
q^{\ast }}{\alpha }\right) ^{\beta }\right] ^{\frac{1}{q^{\ast }-1}},
\end{equation*}%
the stochastic processes $A_{n}\left( t\right) $ defined by (\ref{2.2}) for
the mother process as in condition \textbf{B$^{\prime \prime \prime }$}
converge in $\mathcal{L}_{q}$ to the stochastic process $A(t)$ as $%
n\rightarrow \infty ,$ such that $A\left( t\right) \in \mathcal{L}_{q}$ and 
\begin{equation*}
\mathrm{E}A(t)^{q}\sim t^{\varsigma \left( q\right) },
\end{equation*}%
where the scaling function $\varsigma \left( q\right) $ is given by 
\begin{equation*}
\varsigma \left( q\right) =q\left( 1+\frac{1}{\log b}\log \frac{1}{\left( 1-%
\frac{1}{\alpha }\right) ^{\beta }}\right) +\frac{\beta }{\log b}\log \left(
1-\frac{q}{\alpha }\right) ,q\in Q.
\end{equation*}%
Moreover, 
\begin{equation*}
\mathrm{Var}A\left( t\right) \geqslant \int_{0}^{t}\int_{0}^{t}\left(
M(1,1;u-w)-1\right) dudw,
\end{equation*}%
where $M$ is given by (\ref{4.4}) or (\ref{4.5}).
\end{theorem}

\begin{proof}
Theorem~\ref{thm:G} follows from the Theorem~\ref{thm:ou}
\end{proof}

Note that the Theorem~{\ref{thm:G}} is an extension of the log-gamma
scenario in the Theorem 3 of Anh, Leonenko and Shieh (2008a), in which the
set $Q=\{q:0<q\leq 2,\alpha =2,\beta >0\}.$

\bigskip We can construct log-tempered stable scenarios for a more general
class of finite superpositions of stationary gamma OU-type processes (\ref%
{3.12}), where $X_{j}(t),j=1,...,m,$ are independent stationary processes
with marginals $\Gamma (\beta _{j},\alpha ),j=1,...,m.$ Using notation (\ref%
{FS}), we consider the class of processes 
\begin{equation*}
\mathbb{FS}_{m}\{\Gamma (\sum_{j=1}^{m}\beta _{j}),\alpha );\frac{\beta _{j}%
}{\alpha };\frac{\beta _{j}}{\alpha ^{2}}).
\end{equation*}

Theorem~{\ref{thm:G}} can be reformulated for the process of superposition $%
X_{m\sup }$ with $\beta =\sum_{j=1}^{m}\beta _{j}$ and 
\begin{align*}
M\left( \zeta _{1},\zeta _{2};\left( t_{1}-t_{2}\right) \right) & =\mathrm{E}%
\exp \left\{ \zeta _{1}\left( X_{m\sup }(t_{1}\right) -c_{X})+\zeta
_{2}\left( X_{m\sup }(t_{2}\right) -c_{X})\right\} \\
& =e^{-c_{X}\left( \zeta _{1}+\zeta _{2}\right) }\mathrm{E}\exp \left\{
\zeta _{1}X_{m\sup }(t_{1})+\zeta _{2}\left( X_{m\sup }(t_{2}\right)
\right\} ,
\end{align*}%
where%
\begin{equation*}
\log \mathrm{E}\exp \{\zeta _{1}X_{m\sup }(t_{1})+\zeta _{2}X_{m\sup
}(t_{2})\}=\sum_{j=1}^{m}\log \mathrm{E}\exp \{\zeta _{1}X_{j}(t_{1})+\zeta
_{2}X_{j}(t_{2})\},
\end{equation*}%
and $\log \mathrm{E}\exp \{\zeta _{1}X_{j}(t_{1})+\zeta
_{2}X_{j}(t_{2})\},j=1,..,m$ are given by (\ref{4.4}) or (\ref{4.5}).

Moreover, one can construct log-gamma scenarios for a more general class of
infinite superpositions of stationary gamma OU-type processes (\ref{sup1}),
where $X_{j}(t),j=1,...,$ are independent stationary processes with
marginals $\Gamma (\beta _{j},\alpha ),j=1,2..$. Using notation (\ref{IS}),
we consider the class of processes 
\begin{equation*}
\mathbb{IS}\{\Gamma ((\sum_{j=1}^{\infty }\beta _{j}),\alpha );\frac{\beta
_{j}}{\alpha };\frac{\beta _{j}}{\alpha ^{2}})\}.
\end{equation*}

Then the statement of the Theorem~{\ref{thm:G}} remains true with $\beta
=\sum_{j=1}^{\infty }\beta _{j}$ and 
\begin{align*}
M\left( \zeta _{1},\zeta _{2};\left( t_{1}-t_{2}\right) \right) & =\mathrm{E}%
\exp \left\{ \zeta _{1}\left( X_{\sup }(t_{1}\right) -c_{X})+\zeta
_{2}\left( X_{\sup }(t_{2}\right) -c_{X})\right\} \\
& =e^{-c_{X}\left( \zeta _{1}+\zeta _{2}\right) }\mathrm{E}\exp \left\{
\zeta _{1}X_{\sup }(t_{1})+\zeta _{2}\left( X_{\sup }(t_{2}\right) \right\} ,
\end{align*}%
where%
\begin{equation*}
\log \mathrm{E}\exp \{\zeta _{1}X_{\sup }(t_{1})+\zeta _{2}X_{\sup
}(t_{2})\}=\sum_{j=1}^{\infty }\log \mathrm{E}\exp \{\zeta
_{1}X_{j}(t_{1})+\zeta _{2}X_{j}(t_{2})\},
\end{equation*}%
and $\log \mathrm{E}\exp \{\zeta _{1}X_{j}(t_{1})+\zeta
_{2}X_{j}(t_{2})\},j=1,2,...$ are given by (\ref{4.4}) or (\ref{4.5}).

\section{Log-variance gamma scenario}

The next example of a hyperbolic OU process is based on the variance-gamma
distribution (see, for example, Madan \emph{et al.} 1998, Finlay and Seneta
2006, Carr \emph{et al.} 2007). We will use a stationary OU type process (%
\ref{3.4}) with marginal variance gamma distribution $VG\left( \kappa
,\alpha ,\beta ,\mu \right) $, which has the moment generating function 
\begin{equation*}
\log \mathrm{E}e^{\zeta X(t)}=\mu \zeta +2\kappa \log \left( \gamma /\sqrt{%
\alpha ^{2}-\left( \beta +\zeta \right) ^{2}}\right) ,\quad \left\vert \beta
+\zeta \right\vert <\alpha ,
\end{equation*}%
where the set of parameters is of the form 
\begin{equation*}
\gamma ^{2}=\alpha ^{2}-\beta ^{2},\kappa >0,\alpha >\left\vert \beta
\right\vert >0,\mu \in \mathbb{R}.
\end{equation*}%
It is known that this distribution is self-decomposable. Note that%
\begin{equation*}
\mathrm{E}X\left( t\right) =\mu +\frac{2\beta \kappa }{\gamma ^{2}},\quad 
\mathrm{Var}X\left( t\right) =\frac{2\kappa }{\gamma ^{2}}\left( 1+2\left( 
\frac{\beta }{\gamma }\right) ^{2}\right) .
\end{equation*}%
Thus, if $X_{j,}(t)\,j=1,...,m,$ are independent so that $X_{j}\thicksim
VG\left( \kappa _{j},\alpha ,\beta ,\mu _{j}\right) ,$ $j=1,...,m,$ then we
have that 
\begin{equation*}
X_{1}(t)+...+X_{m}(t)\thicksim VG\left( \kappa _{1}+...+\kappa _{n},\alpha
,\beta ,\mu _{1}+...+\mu _{n}\right) .
\end{equation*}%
The L\'{e}vy measure $\nu $ of $X(t)$ has density%
\begin{equation}
p\left( u\right) =\frac{\kappa }{\left\vert u\right\vert }e^{\beta u-\alpha
\left\vert u\right\vert },u\in \mathbb{R}.  \label{eq4171}
\end{equation}%
By (\ref{3.6}) the L\'{e}vy measure $\tilde{\nu}$ of the BDLP $Z(t)$ in (\ref%
{3.4}) has density%
\begin{equation}
q\left( u\right) =-p\left( u\right) -up^{\prime }\left( u\right) ,
\label{eq418}
\end{equation}%
\begin{equation*}
p^{\prime }\left( u\right) =\left\{ 
\begin{array}{cc}
-\frac{\kappa }{u}e^{u(\beta +\alpha )}(\beta +\alpha )+\frac{\kappa }{u^{2}}%
e^{u(\beta +\alpha )}, & u<0, \\ 
\frac{\kappa }{u}e^{u(\beta -\alpha )}(\beta -\alpha )-\frac{\kappa }{u^{2}}%
e^{u(\beta -\alpha )}, & u>0.%
\end{array}%
\right.
\end{equation*}

Consider a mother process of the form 
\begin{equation*}
\Lambda (t)=\exp \left( X\left( t\right) -c_{X}\right) ,c_{X}=\mu +2\kappa
\log \left( \gamma /\sqrt{\alpha ^{2}-\left( \beta +1\right) ^{2}}\right)
,\quad \left\vert \beta +1\right\vert <\alpha ,
\end{equation*}%
where $X\left( t\right) $ is a stationary $VG\left( \kappa ,\alpha ,\beta
,\mu \right) $ OU type process with covariance function 
\begin{equation*}
R_{X}\left( t\right) =\frac{2\kappa }{\gamma ^{2}}\left( 1+2\left( \frac{%
\beta }{\gamma }\right) ^{2}\right) \exp \left( -\lambda \left\vert
t\right\vert \right) .
\end{equation*}
\ We obtain the moment generating function%
\begin{equation}
M\left( \zeta \right) =\mathrm{E}\exp \left( \zeta \left( X\left( t\right)
-c_{X}\right) \right) =e^{-c_{X}\zeta }e^{\mu \zeta +2\kappa \log \left(
\gamma /\sqrt{\alpha ^{2}-\left( \beta +\zeta \right) ^{2}}\right) },\text{%
\quad }\,\left\vert \beta +\zeta \right\vert <\alpha ,  \label{4.19}
\end{equation}%
and the bivariate moment generating function 
\begin{align}
M\left( \zeta _{1},\zeta _{2};\left( t_{1}-t_{2}\right) \right) & =\mathrm{E}%
\exp \left( \zeta _{1}\left( X(t_{1}\right) -c_{X})+\zeta _{2}\left(
X(t_{2}\right) -c_{X})\right)  \notag \\
& =e^{-c_{X}\left( \zeta _{1}+\zeta _{2}\right) }\mathrm{E}\exp \left( \zeta
_{1}X(t_{1})+\zeta _{2}\left( X(t_{2}\right) \right) ,  \label{4.20}
\end{align}%
where $\mathrm{E}\exp \left( \zeta _{1}X(t_{1})+\zeta _{2}\left(
X(t_{2}\right) \right) $ is given by (\ref{eq39}) with L\'{e}vy measure $%
\tilde{\nu}$ \ having density (\ref{eq418}). Thus, the correlation function
of the mother process takes the form 
\begin{equation*}
\rho (\tau )=\frac{M(1,1;\tau )-1}{M(2)-1},
\end{equation*}%
where $M(2)$ is given by (\ref{4.19}) and $M(1,1;\tau )$ is given by (\ref%
{4.20}).

The condition (\ref{eq:existence.derivative}) of Theorem \textit{\ref{thm:ou}%
} holds for $q<\alpha -\left\vert \beta \right\vert $ $.$

\begin{theorem}
\label{thm:VG}Let $X\left( t\right) $ be a stationary $VG\left( \kappa
,\alpha ,\beta ,\mu \right) $ OU type process and let 
\begin{equation*}
Q=\left\{ q:0<q<q^{\ast }\,<\left\vert \alpha \right\vert -\left\vert \beta
\right\vert ,\kappa >0\right\} ,
\end{equation*}%
where $q^{\ast }$ is a fixed integer.

Then, for any 
\begin{equation*}
b>\exp \left\{ 2\kappa \left[ \frac{1}{1-q^{\ast }}\log \frac{\gamma }{\sqrt{%
\alpha ^{2}-(\beta +q^{\ast })^{2}}}+\frac{q^{\ast }}{1-q^{\ast }}\log \frac{%
\gamma }{\sqrt{\alpha ^{2}-(\beta +1)^{2}}}\right] \right\} ,
\end{equation*}%
the stochastic processes $A_{n}\left( t\right) $ defined by (\ref{2.2}) for
the mother process as in condition \textbf{B}$^{\prime \prime \prime \prime
} $ converge in $\mathcal{L}_{q}$ to the stochastic process $A(t)$ as $%
n\rightarrow \infty $ such that, if $A\left( 1\right) \in \mathcal{L}_{q}$
for $q\in Q,$%
\begin{equation*}
\mathrm{E}A(t)^{q}\sim t^{\varsigma \left( q\right) },
\end{equation*}%
where the scaling function is given by 
\begin{equation*}
\varsigma \left( q\right) =q\left( 1+\frac{2\kappa }{\log b}\log \frac{%
\gamma }{\sqrt{\alpha ^{2}-(\beta +1)^{2}}}\right) +\frac{2\kappa }{\log b}%
\log \sqrt{\alpha ^{2}-\left( \beta +q\right) ^{2}}-\frac{2\kappa }{\log b}%
\log \gamma .
\end{equation*}%
Moreover, 
\begin{equation*}
\mathrm{Var}A\left( t\right) \geqslant \int_{0}^{t}\int_{0}^{t}\left(
M(1,1;u-w)-1\right) dudw,
\end{equation*}%
where $M$ is given by (\ref{4.20}).
\end{theorem}

\begin{proof}
Theorem~{\ref{thm:VG}} follows from the Theorem~\ref{thm:ou}.
\end{proof}

\bigskip

\bigskip We can construct log-variance-gamma scenarios for a more general
class of finite superpositions of stationary variance gamma OU-type
processes (\ref{3.12}), where $X_{j}(t),j=1,...,m,$ are independent
stationary processes with marginals $VG\left( \kappa _{j},\alpha ,\beta ,\mu
_{j}\right) ,j=1,...,m.$ Using notation (\ref{FS}), we consider the class of
processes 
\begin{equation*}
\mathbb{FS}_{m}\{VG\left( \kappa _{1}+...+\kappa _{m},\alpha ,\beta ,\delta
,\mu _{1}+...+\mu _{m}\right) ;\mu _{j}+\frac{2\beta }{\gamma ^{2}}\kappa
_{j};\frac{2}{\gamma ^{2}}\left( 1+2\left( \frac{\beta }{\gamma }\right)
^{2}\right) \kappa _{j}\}.
\end{equation*}

The generalization of Theorem~{\ref{thm:VG}} \ remains true for this
situation with $\kappa =\sum_{j=1}^{m}\kappa _{j}$ , $\mu =\sum_{j=1}^{m}\mu
_{j},$ and 
\begin{align*}
M\left( \zeta _{1},\zeta _{2};\left( t_{1}-t_{2}\right) \right) & =\mathrm{E}%
\exp \left\{ \zeta _{1}\left( X_{m\sup }(t_{1}\right) -c_{X})+\zeta
_{2}\left( X_{m\sup }(t_{2}\right) -c_{X})\right\} \\
& =e^{-c_{X}\left( \zeta _{1}+\zeta _{2}\right) }\mathrm{E}\exp \left\{
\zeta _{1}X_{m\sup }(t_{1})+\zeta _{2}\left( X_{m\sup }(t_{2}\right)
\right\} ,
\end{align*}%
where%
\begin{equation*}
\log \mathrm{E}\exp \{\zeta _{1}X_{m\sup }(t_{1})+\zeta _{2}X_{m\sup
}(t_{2})\}=\sum_{j=1}^{m}\log \mathrm{E}\exp \{\zeta _{1}X_{j}(t_{1})+\zeta
_{2}X_{j}(t_{2})\},
\end{equation*}%
and $\log \mathrm{E}\exp \{\zeta _{1}X_{j}(t_{1})+\zeta
_{2}X_{j}(t_{2})\},j=1,..,m$ are given by (\ref{4.20}).

We can construct log-variance-gamma scenarios for a more general class of
infinite superpositions of stationary variance gamma OU-type processes (\ref%
{sup1}), where $X_{j}(t),j=1,...,$ are independent stationary processes with
marginals $VG\left( \kappa _{j},\alpha ,\beta ,\mu _{j}\right) .$ Using
notation (\ref{IS}), we consider the class of processes%
\begin{equation*}
\mathbb{IS\{}VG\left( \sum_{j=1}^{\infty }\kappa _{j},\alpha ,\beta
,\sum_{j=1}^{\infty }\mu _{j}\right) ;\mu _{j}+\frac{2\beta }{\gamma ^{2}}%
\kappa _{j};\frac{2}{\gamma ^{2}}\left( 1+2\left( \frac{\beta }{\gamma }%
\right) ^{2}\right) \kappa _{j}\}.
\end{equation*}

Then the statement of the Theorem 15 remains true with $\mu
=\sum_{j=1}^{\infty }\mu _{j},\kappa =\sum_{j=1}^{\infty }\kappa _{j},$ and 
\begin{align*}
M\left( \zeta _{1},\zeta _{2};\left( t_{1}-t_{2}\right) \right) & =\mathrm{E}%
\exp \left\{ \zeta _{1}\left( X_{\sup }(t_{1}\right) -c_{X})+\zeta
_{2}\left( X_{\sup }(t_{2}\right) -c_{X})\right\} \\
& =e^{-c_{X}\left( \zeta _{1}+\zeta _{2}\right) }\mathrm{E}\exp \left\{
\zeta _{1}X_{\sup }(t_{1})+\zeta _{2}\left( X_{\sup }(t_{2}\right) \right\} ,
\end{align*}%
where%
\begin{equation*}
\log \mathrm{E}\exp \{\zeta _{1}X_{\sup }(t_{1})+\zeta _{2}X_{\sup
}(t_{2})\}=\sum_{j=1}^{\infty }\log \mathrm{E}\exp \{\zeta
_{1}X_{j}(t_{1})+\zeta _{2}X_{j}(t_{2})\},
\end{equation*}%
and $\log \mathrm{E}\exp \{\zeta _{1}X_{j}(t_{1})+\zeta
_{2}X_{j}(t_{2})\},j=1,..,...$ are given by (\ref{4.20}).

\section{Connections and prospects}

Both papers Musy and Barcy (2002) and Barral and Jin (2012) (see also their
references) introduce multifractal random measures $\mu $ as a limit of
positive martingales $\mu _{j\text{ }}$defined in a framework of
log-infinitely divisible cascades constructed as independently scattered
random measures on some cones on the plane. In particularly, Barral and Jin
(2012) extended some classical results valid for canonical multiplicative
cascades to exact scaling of log-infinitely divisible cascades.

If $\psi (z)$ is the characteristic L\'{e}vy exponent with L\'{e}vy triplet $%
(a,d,\nu )$, see (\ref{eq:Levy-Khinchine.exponent.cond}), and using notation
of the paper, let $\varphi (q)=\log _{2}\mathrm{E}(W^{q})-(q-1)=\psi (-%
\mathrm{i}q)-(q-1),$ for some infinitely divisible random variable $\ W$,
which generates cascade, then

(i) the necessary and sufficient condition for non-generacy of $\mu $, is of
the form: $\varphi ^{\prime }(1^{-})<0,$ and (ii) the necessary and
sufficient condition for $\mathrm{E}$ $(\left\Vert \mu \right\Vert
^{q})<\infty ,$ is of the form: $\varphi (q)<0,q>1.$ Also, if $\psi (-%
\mathrm{i}2)<\infty ,$ the increments of limiting multifractal measure is
stationary process with long-range dependence, see again Barral and Jin
(2012).

Bacry and Muzy's construction uses other shapes for the cone, but in the
notation above the condition of non-generacity is of the form i), while the
condition of $L_{2}$ convergence and $\mathrm{E}$ $(\left\Vert \mu
\right\Vert ^{2})<\infty ,$ is of the form: $\bar{\psi}(2)<1,$ where $\bar{%
\psi}(q)$ is the Laplace exponent of L\'{e}vy-Khintchin representation of
some infinitely divisible random variable. In this case no long-range
dependence between the increments of the multifractal measure, and in order
to have long-range dependence they used the so-called multifractal random
walk, that is superposition of fractional Brownian \ \ \ motion and limiting
multifractal process, assuming that they are independent.

Our construction has connection to both papers. Firstly, we present general
results on $\mathcal{L}_{q}$ convergence (Theorems 1 and 2) without any
assumptions about log-infinitely divisibility of mother process. These
results are more general then results of the above papers. To see this, one
can apply these results (for $q=2)$ for the geometric stationary diffusion
mother process, in which cases several scenarios are possible, including
log-beta scenario, which is not log-infinitely divisible, see [6] for more
details. Both short-range dependence and long-range dependence potentially
covered by Theorems 1 and 2. Then we consider the geometric OU processes,
which have log-self-decomposable marginal distributions, this is a subclass
of log-infinitely divisible distributions, and inclusion is strict. In this
case our results are less general in terms of possible scenarios, as well as
our conditions, see Theorem 4. In particular our condition for log-gamma
scenario required $\alpha >2$ (see Theorem 9), while in the framework of the
paper Musy and Barcy (2002) for log-gamma scenario one needs only $\alpha
>1. $ Also, for a $\alpha $-stable OU process, the results Musy and Barcy
(2002) and Barral and Jin (2012) hold for $\alpha \in (0,2),$ while our
condition (6.12) does not hold. Next, the results of the papers Musy and
Barcy (2002) and Barral and Jin (2012) can be applied for discrete
infinitely divisible distributions, i.e., to get log-poisson scenario, while
our results of sections 6 and 7 can not be applied for discrete
distributions, since they are not self-decomposable. However by using
results of sections 2 and 3, one can obtain log-poisson scenarios (among the
others) by using the multiplicative products of ergodic birth-death
processes, see [5] for details. As far as dependence is concern our approach
allows to model both short- and long-range dependence, see sections 6 and 7.

In the same spirit one can obatin the log Meixner or more generally log-z
multifractal scenario (see Anh, Leonenko and Shieh 2008a) or log-Euler's
gamma multifractal scenario (see Anh, Leonenko and Shieh 2008b). In
principle, it is possible to obtain the log-hyperbolic scenarios for which
there exist exact forms of L\'{e}vy measures of the OU process and the BDLP L%
\'{e}vy process; however some analytical work is still to be carried out.
This will be done elsewhere.

\section*{Acknowledgements}

The authors wish to thank two referees and the associate editor for the
comments and suggestions which have led to the improvement of the manuscript.%
\newline

{\normalsize N.N. Leonenko partially supported by grant of the European
commission PIRSES-GA-2008-230804 (Marie Curie), projects MTM2009-13393, by
projects MTM2012-32674 of the Ministry of Education and Science of Spain,
and P09-FQM-5052 of the Andalousian CICE, Spain.}

\bigskip

\end{document}